\newif\ifsiopt

\sioptfalse 

\ifsiopt
	\documentclass[review,nohypdvips]{siamart1116}
\else
	\RequirePackage{fix-cm}
	\documentclass[smallextended]{svjour3} 
	\smartqed
\fi

\usepackage{graphicx}
\usepackage{amsmath,amssymb}
\usepackage{array}
\usepackage{enumitem}
\usepackage{xcolor}
\usepackage{bm}
\usepackage{multirow}
\providecommand{\etal}		{\emph{et al\@.}\xspace}
\providecommand{\ie}		{\emph{i.e\@.}\xspace}
\providecommand{\eg}		{\emph{e.g\@.}\xspace}

\providecommand{\myurl}[1][]	{\texttt{web.eecs.umich.edu/$\sim$fessler#1}\xspace}

\providecommand{\onweb}[1]	{Available from \myurl.}

\long\def\comment#1{}

\providecommand{\bcent}		{\begin{center}}
\providecommand{\ecent}		{\end{center}}
\providecommand{\benum}		{\begin{enumerate}}
\providecommand{\eenum}		{\end{enumerate}}
\providecommand{\bitem}		{\begin{itemize}}
\providecommand{\eitem}		{\end{itemize}}
\providecommand{\bvers}		{\begin{verse}}
\providecommand{\evers}		{\end{verse}}
\providecommand{\btab}		{\begin{tabbing}}	
\providecommand{\etab}		{\end{tabbing}}

\newcounter{blist}
\providecommand{\blistmark}	{\makebox[0pt]{$\bullet$}}
\providecommand{\blistitemsep}	{0pt}
\providecommand{\blist}[1][]	{%
\begin{list}{\blistmark}{%
\usecounter{blist}%
\setlength{\itemsep}{\blistitemsep}%
\setlength{\parsep}{0pt}%
\setlength{\parskip}{0pt}%
\setlength{\partopsep}{0pt}%
\setlength{\topsep}{0pt}%
\setlength{\leftmargin}{1.2em}%
\setlength{\labelsep}{0.5\leftmargin}
\setlength{\labelwidth}{0em}%
#1}
}
\providecommand{\elist}		{\end{list}}

\providecommand{\blistitemsep}	{0pt}
\providecommand{\bjfenum}[1][]	{%
\begin{list}{\bcolor{\arabic{blist}.} }{%
\usecounter{blist}%
\setlength{\itemsep}{\blistitemsep}%
\setlength{\parsep}{0pt}%
\setlength{\parskip}{0pt}%
\setlength{\partopsep}{0pt}%
\setlength{\topsep}{0pt}%
\setlength{\leftmargin}{0.0em}%
\setlength{\labelsep}{1.0\leftmargin}
\setlength{\labelwidth}{0pt}%
#1}
}

\newcounter{blistAlph}
\providecommand{\blistAlph}[1][]
{\begin{list}{\makebox[0pt][l]{\Alph{blistAlph}.}}{%
\usecounter{blistAlph}%
\setlength{\itemsep}{0pt}\setlength{\parsep}{0pt}%
\setlength{\parskip}{0pt}\setlength{\partopsep}{0pt}%
\setlength{\topsep}{0pt}%
\setlength{\leftmargin}{1.2em}%
\setlength{\labelsep}{1.0\leftmargin}
\setlength{\labelwidth}{0.0\leftmargin}#1}%
}

\newcounter{blistRoman}
\providecommand{\blistRoman}[1][]
{\begin{list}{\Roman{blistRoman}.}{%
\usecounter{blistRoman}%
\setlength{\itemsep}{0.5em}\setlength{\parsep}{0pt}%
\setlength{\parskip}{0pt}\setlength{\partopsep}{0pt}%
\setlength{\topsep}{0pt}%
\setlength{\leftmargin}{4em}%
\setlength{\labelsep}{0.4\leftmargin}
\setlength{\labelwidth}{0.6\leftmargin}#1}%
}

\usepackage{bbm} 

\providecommand{\norm}[1]	{\xmath{\left\| #1 \right\|}}

\providecommand{\floor}[1]	{\xmath{\left\lfloor #1 \right\rfloor}}

\providecommand{\inprod}[2]	{\xmath{\mathop{\langle #1,\, #2 \rangle}\nolimits}}
\providecommand{\Inprod}[2]	{\xmath{\left\langle #1,\ #2 \right\rangle}}

\let\equivsave\equiv
\def\equiv{\xmath{\equivsave}}

\providecommand{\ba}[1]		{\left[ \begin{array}{#1}}
\providecommand{\ea}		{\end{array} \right]}
\providecommand{\be}		{\begin{equation}}
\providecommand{\ee}[1]		{\label{#1}\end{equation}}
\providecommand{\bea}		{\begin{eqnarray}}
\providecommand{\eea}[1]	{\label{#1}\end{eqnarray}}
\providecommand{\beas}		{\begin{eqnarray*}}
\providecommand{\eeas}		{\end{eqnarray*}}
\providecommand{\beals}[1][1]	{\begin{alignat*}{#1}}	
\providecommand{\eeals}		{\end{alignat*}}

\providecommand{\berr}[2]{
\bgroup
\renewcommand{\theequation}{#1}
\be
#2
\ee{e,#1}
\egroup
\ignorespaces
}

\providecommand{\bearr}[2]{
\bgroup
\renewcommand{\theequation}{#1}
\bea
#2
\eea{e,#1}
\egroup
\ignorespaces
}

\providecommand{\inmath}	{\ensuremath}
\providecommand{\xmath}[1]	{\inmath{#1}\xspace}
\providecommand{\bmath}[1]	{\xmath{\bm{#1}}}	

\providecommand{\paren}[1]	{\xmath{\left(#1\right)}}

\providecommand{\braces}[1]	{\xmath{\left\{#1\right\}}}

\providecommand{\xfrac}[2]	{\xmath{\frac{#1}{#2}}}

\providecommand{\Frac}[2]	{\xmath{{#1}/{#2}}}
\providecommand{\pfrac}[3][]	{\xmath{\!\paren{#1\frac{#2}{#3}}}}

\providecommand{\wordbrace}[3][]{\,\inmath{\mathrm{#2}#1\!\braces{#3}}}

\providecommand{\diag}[1]	{\wordbrace{diag}{#1}}

\providecommand{\sinc}		{\Ofun{\mathrm{sinc}}}

\usepackage{jf-accent}
\usepackage{jf-fun}
\usepackage{jf-names}

\newcolumntype{L}{>{\arraybackslash}p{4.4em}}
\newcommand{\st} {\xmath{\text{s.t.}\:}}
\newcommand{\tnabla} {\xmath{\tilde{\nabla}_{\!L}}}
\newcommand{\tnablasig} {\xmath{\tilde{\nabla}_{\!\Frac{L}{\sig^2}}}}
\newcommand{\fNh} {\xmath{\floor{\frac{N}{2}}}}
\newcommand{\fNm} {\xmath{\floor{\frac{2N}{3}}}}

\renewcommand{\Om} {\xmath{\Omega_N}}

\newcommand{\bmw} {\bmath{w}}
\newcommand{\bmlam} {\bmath{\lambda}}
\newcommand{\bmtau} {\bmath{{\tau}}}
\newcommand{\bmh} {\bmath{h}}
\newcommand{\g} {\bmath{g}}
\newcommand{\G} {\bmath{G}}
\newcommand{\bmdel} {\bmath{\delta}}
\newcommand{\bmu} {\bmath{u}}
\newcommand{\bmv} {\bar{\bmu}}
\newcommand{\bme} {\bmath{e}}
\newcommand{\x} {\bmath{x}}
\newcommand{\y} {\bmath{y}}
\newcommand{\z} {\bmath{z}}
\newcommand{\bmnu} {\bmath{\nu}}

\newcommand{\bmS} {\bmath{S}}

\newcommand{\A} {\bmath{\check{A}}}
\newcommand{\D} {\bmath{\check{D}}}
\newcommand{\bmbeta} {\bmath{{\beta}}}

\newcommand{\tG} {\bar{\G}}
\newcommand{\tg} {\bar{\g}}
\newcommand{\bA} {\bmath{\bar{A}}}
\newcommand{\bD} {\bmath{\bar{D}}}
\newcommand{\Zero} {\bmath{0}}

\renewcommand{\tt} {\bmath{t}}
\newcommand{\ttt} {\bar{\tt}}
\newcommand{\TT} {\bmath{T}}

\newcommand{\cC} {\xmath{\mathcal{C}_L^{1,1}(\Reals^d)}}
\newcommand{\cF} {\xmath{\mathcal{F}_L(\Reals^d)}}
\newcommand{\cL} {\xmath{\mathcal{L}}}
\newcommand{\Reals} {\xmath{\mathbb{R}}}

\newcommand{\hki} {\xmath{h_{i,k}}}
\newcommand{\hkip} {\xmath{h_{i+1,k}}}
\newcommand{\himi} {\xmath{h_{i,i-1}}}
\newcommand{\hkj} {\xmath{h_{j,k}}}

\newcommand{\hiip} {\xmath{h_{i+1,i}}}
\newcommand{\himip} {\xmath{h_{i+1,i-1}}}

\newcommand{\hkn} {\xmath{h_{n,k}}}
\newcommand{\hknp} {\xmath{h_{n+1,k}}}

\newcommand{\pL} {\xmath{\bm{\mathrm{p}}_L^{ }}}
\newcommand{\pLsig} {\xmath{\bm{\mathrm{p}}_{\Frac{L}{\sigma^2}}^{ }}}
\renewcommand{\P} {\xmath{\mathrm{P}}}
\newcommand{\I} {\xmath{\mathrm{I}}}

\newcommand{\FO} {{FSFOM}\xspace}
\newcommand{\cgm} {{composite gradient mapping}\xspace}
\newcommand{\m} {{\xmath{m}}\xspace}
\newcommand{\OPG} {{OCG}\xspace}

\usepackage{hyphenat}

\numberwithin{equation}{section}
\ifsiopt
	\numberwithin{theorem}{section}
\else
	\usepackage{fullpage}
\fi
\usepackage{hyperref}
\hypersetup{
        linktoc=page,
        colorlinks=true,
        linkcolor=blue,
        citecolor=red,
        filecolor=magenta,
        urlcolor=cyan
}

\newcommand{\TheTitle}{Another look at the fast iterative shrinkage/thresholding 
			algorithm (FISTA)}
\newcommand{\TheAuthors}{D. Kim, and J. A. Fessler}

\ifsiopt
	
	\headers{Another look at the FISTA}{\TheAuthors}
	
	\title{{\TheTitle}\thanks{Submitted to the editors \today.
	\funding{This research was supported in part by NIH grant U01 EB018753.}}}

	\author{
  	Donghwan Kim\thanks{Dept. of Electrical Engineering and Computer Science,
        	University of Michigan, Ann Arbor, MI 48109 USA
       		(\email{kimdongh@umich.edu}, \email{fessler@umich.edu})}
	\and 
	Jeffrey A. Fessler\footnotemark[2] 
	}

	\ifpdf
	\hypersetup{
	pdftitle={\TheTitle}
	pdfauthor={\TheAuthors}
	}
	\fi

\else
	\title{\TheTitle
		\thanks{This research was supported in part by NIH grant U01 EB018753.}}
	\author{Donghwan Kim \and Jeffrey A. Fessler}
	\institute{Donghwan Kim \and Jeffrey A. Fessler \at
              Dept. of Electrical Engineering and Computer Science,
		University of Michigan, Ann Arbor, MI 48109 USA \\
              \email{kimdongh@umich.edu, fessler@umich.edu}         
	}
	\date{Date of current version: \today} 
\fi

\begin{document}

\maketitle

\begin{abstract}
This paper provides a new way of developing
the fast iterative shrinkage/thresh
olding algorithm (FISTA)
\cite{beck:09:afi}
that is widely used for minimizing
composite convex functions
with a nonsmooth term such as the $\ell_1$ regularizer.
In particular, this paper 
shows that FISTA
corresponds to an optimized approach
to accelerating the proximal gradient method
with respect to a worst-case bound of the cost function.
This paper then proposes 
a new algorithm
that is derived
by instead optimizing the step coefficients of the proximal gradient method 
with respect to a worst-case bound of the~\cgm.
The proof is based on the worst-case analysis
called Performance Estimation Problem in~\cite{drori:14:pof}.
\end{abstract}

\section{Introduction}
\label{intro}

The fast iterative shrinkage/thresholding algorithm 
(FI
STA)~\cite{beck:09:afi},
also known as a fast proximal gradient method (FPGM) in general,
is a very widely used fast first-order method.
FISTA's speed arises from Nesterov's accelerating technique
in~\cite{nesterov:83:amf,nesterov:04}
that improves the $O(1/N)$ cost function worst-case bound
of a proximal gradient method (PGM)
to the optimal $O(1/N^2)$ rate
where $N$ denotes the number of iterations
\cite{beck:09:afi}.

This paper first provides
a new way to develop
Nesterov's acceleration approach, \ie, FISTA (FPGM).
In particular, we show that FPGM corresponds
to an optimized approach to accelerating PGM
with respect to a worst-case bound of the cost function.
We then propose 
a new fast algorithm that is derived from PGM
by instead optimizing a worst-case bound of the~\cgm.
We call this new method
FPGM-\OPG (\OPG for optimized over~\cgm).
This new method 
provides the best known analytical worst-case bound
for decreasing the~\cgm with rate $O(1/N^{\frac{3}{2}})$
among fixed-step first-order methods.
The proof is based on the worst-case bound analysis
called Performance Estimation Problem (PEP) in~\cite{drori:14:pof},
which we briefly review next.

Drori and Teboulle's PEP~\cite{drori:14:pof}
casts a worst-case analysis 
for a given optimization method and a given class of optimization problems
into a meta-optimization problem.
The original PEP has been intractable to solve exactly,
so~\cite{drori:14:pof} introduced a series of tractable relaxations,
focusing on first-order methods and smooth convex minimization problems;
this PEP and its relaxations were studied
for various algorithms and minimization problem classes in
\cite{drori:16:aov,kim:16:gto,kim:16:ofo,kim:17:otc,lessard:16:aad,taylor:17:ewc,taylor:17:ssc}.
Drori and Teboulle~\cite{drori:14:pof} 
further proposed to optimize the step coefficients
of a given class of optimization methods using a PEP.
This approach was studied for first-order methods
on unconstrained smooth convex minimization problems
in~\cite{drori:14:pof},
and the authors~\cite{kim:16:ofo} derived 
a new first-order method,
called an optimized gradient method (OGM)
that has an analytic worst-case bound on the cost function
that is twice smaller than the previously best known bounds 
of~\cite{nesterov:83:amf,nesterov:04}.
Recently, Drori~\cite{drori:17:tei} showed
that the OGM exactly achieves the optimal cost function worst-case bound 
among first-order methods
for smooth convex minimization (in high-dimensional problems).

Building upon~\cite{drori:14:pof} and its successors,
Taylor~\etal~\cite{taylor:17:ewc}
expanded the use of
PEP to 
first-order (proximal gradient) methods
for minimizing nonsmooth composite convex functions. 
They used a tight relaxation\footnote{
Tight relaxation here denotes
transforming (relaxing) an optimization problem into a solvable problem
while their solutions remain the same.
\cite{taylor:17:ewc}
tightly relaxes the PEP into a solvable equivalent problem
under a large-dimensional condition.
}
for PEP
and studied the tight (exact) 
numerical worst-case bounds
of FPGM, a proximal gradient version of OGM, and some variants
versus number of iterations $N$.
Their numerical results suggest that
there exists an OGM-type acceleration of PGM
that has a worst-case cost function bound
that is about twice smaller than that of FPGM,
showing room for improvement in 
accelerating PGM.
However, 
it is difficult to derive an analytical worst-case bound
for the tightly relaxed PEP in~\cite{taylor:17:ewc},
so optimizing
the step coefficients of PGM
remains an open problem,
unlike~\cite{drori:14:pof,kim:16:ofo}
for smooth convex minimization.

Different from the tightly relaxed PEP in~\cite{taylor:17:ewc},
this paper suggests a new (looser) relaxation of a cost function form of PEP
for nonsmooth composite convex minimization
that simplifies analysis and
optimization of step coefficients of PGM,
although yields loose worst-case bounds.
Interestingly, the resulting optimized PGM
numerically appears to be the FPGM. 
Then, we further provide a new generalized version of FPGM
using our relaxed PEP
that extends our understanding of the FPGM variants.

This paper
next extends
the PEP analysis of the gradient norm
in~\cite{taylor:17:ewc,taylor:17:ssc}.
For unconstrained smooth convex minimization,
the authors~\cite{kim:16:gto} used such PEP to optimize the step coefficients
with respect to the gradient norm.
The corresponding optimized algorithm can be 
useful particularly when dealing with dual problems
where the gradient norm decrease is important
in addition to the cost function minimization
(see \eg, \cite{devolder:12:dst,necoara:16:ica,nesterov:12:htm}).
By extending~\cite{kim:16:gto}, this paper optimizes 
the step coefficients 
of the PGM for
the~\cgm form of PEP for nonsmooth composite convex minimization.
The resulting optimized algorithm differs somewhat from Nesterov's acceleration
and turns out to belong to 
the proposed generalized FPGM class.

Sec.~\ref{sec:prob} describes 
a nonsmooth composite convex minimization problem
and first-order (proximal gradient) methods.
Sec.~\ref{sec:pep,cost}
proposes a new relaxation of PEP
for nonsmooth composite convex minimization problems
and the proximal gradient methods,
and suggests that the FPGM (FISTA)~\cite{beck:09:afi}
is the optimized method of 
the cost function form of this relaxed PEP.
Sec.~\ref{sec:pep,cost} further
proposes a generalized version of FPGM using the relaxed PEP.
Sec.~\ref{sec:pep,spgrad}
studies the~\cgm form of the relaxed PEP
and describes a new optimized method
for decreasing the norm of~\cgm.
Sec.~\ref{sec:disc} 
compares the various algorithms considered,
and Sec.~\ref{sec:conc} concludes.

\section{Problem, methods, and contribution}
\label{sec:prob}

We consider first-order algorithms
for solving the nonsmooth composite convex minimization problem:
\begin{align}
\min_{\x\in\Reals^d} \;&\; \braces{F(\x) := f(\x) + \phi(\x)}
\label{eq:prob} \tag{M}
,\end{align}
under the following assumptions:
\begin{itemize}[leftmargin=40pt]
\item
$f\;:\;\Reals^d\rightarrow\Reals$
is a convex function of the type \cC, i.e., continuously differentiable with
Lipschitz continuous gradient:
\begin{align}
||\nabla f(\x) - \nabla f(\y)|| \le L||\x - \y||, \quad \forall \x, \y \in \Reals^d,
\label{eq:L}
\end{align}
where $L > 0$ is the Lipschitz constant.
\item
$\phi\;:\;\Reals^d\rightarrow\Reals$ is
a proper, closed, convex and 
``proximal-friendly''~\cite{combettes:11:psm} function.
\item
The optimal set $X_*(F)=\argmin{\x\in\Reals^d} F(\x)$ is nonempty,
\ie, the problem~\eqref{eq:prob} is solvable.
\end{itemize}
We use \cF to denote the class of functions $F$ that satisfy the above conditions.
We additionally assume that
the distance between the initial point $\x_0$
and an optimal solution $\x_* \in X_∗(F)$ is bounded by
$R > 0$, i.e., $||\x_0 - \x_*|| \le R$.

PGM is a standard first-order method 
for solving the problem~\eqref{eq:prob}~\cite{beck:09:afi,combettes:11:psm},
particularly when the following proximal gradient update
(that consists of a gradient descent step 
and a proximal operation~\cite{combettes:11:psm})
is relatively simple:
\begin{align}
\pL(\y)
        &:= \argmin{\x} \braces{f(\y) + \inprod{\x - \y}{\nabla f(\y)}
		+ \frac{L}{2}||\x - \y||^2 + \phi(\x)}
		\label{eq:pm} \\
	&= \argmin{\x} \braces{\frac{L}{2}\norm{\x
                - \paren{\y - \frac{1}{L}\nabla f(\y)}}^2 + \phi(\x)}
	\nonumber
.\end{align}
For $\phi(\x) = ||\x||_1$,
the update~\eqref{eq:pm} 
becomes a simple shrinkage/thresholding update,
and PGM reduces to 
an iterative shrinkage/thresholding algorithm (ISTA)~\cite{daubechies:04:ait}.
(See~\cite[Table 10.2]{combettes:11:psm} for more functions $\phi(\x)$ 
that lead to simple proximal operations.)

\fbox{
\begin{minipage}[t]{0.85\linewidth}
\vspace{-10pt}
\begin{flalign}
&\quad \text{\bf Algorithm PGM} & \nonumber \\
&\quad \text{Input: } f\in \cF,\; \x_0\in\Reals^d. & \nonumber \\
&\quad \text{For } i = 0,\ldots,N-1 & \nonumber \\
&\quad \qquad \x_{i+1} = \pL(\x_i) & \nonumber
\end{flalign}
\end{minipage}
} \vspace{5pt}

\noindent
PGM has the following bound 
on the cost function
\cite[Thm. 3.1]{beck:09:afi}
for any $N\ge1$:
\begin{align}
F(\x_N) - F(\x_*)
        \le \frac{LR^2}{2N}
\label{eq:fv_pgm}
.\end{align}

For simplicity in later derivations,
we use the following definition of the~\cgm~\cite{nesterov:13:gmf}:
\begin{align}
\tnabla F(\x) := -L\,(\pL(\x) - \x)
.\end{align}
The~\cgm reduces to 
the usual function gradient $\nabla f(\x)$ when $\phi(\x) = 0$.
We can then rewrite the PGM update in the following form
reminiscent of a gradient method:
\begin{align}
\x_{i+1} = \pL(\x_i) = \x_i - \frac{1}{L}\tnabla F(\x_i)
,\end{align}
where each update guarantees 
the following monotonic cost function descent
\cite[Thm.~1]{nesterov:13:gmf}:
\begin{align}
F(\x_i) - F(\x_{i+1}) \ge \frac{1}{2L}||\tnabla F(\x_i)||^2
\label{eq:pg_decr}
.\end{align}

For any $\x\in\Reals^d$, 
there exists a subgradient $\phi'(\pL(\x)) \in \partial \phi(\pL(\x))$
that satisfies the following equality~\cite[Lemma 2.2]{beck:09:afi}:
\begin{align}
\tnabla F(\x) = \nabla f(\x) + \phi'(\pL(\x))
\label{eq:tnabla}
.\end{align}
This equality implies that
any point $\bar{\x}$ 
with a zero~\cgm
($\tnabla F(\bar{\x}) = 0$, \ie, $\bar{\x} = \pL(\bar{\x})$)
satisfies $0 \in \partial F(\bar{\x})$
and is a minimizer of~\eqref{eq:prob}.
As discussed, minimizing the~\cgm
is noteworthy
in addition to decreasing the cost function.
This property becomes particularly important
when dealing with dual problems.
In particular, it is known that
the norm of the dual (sub)gradient 
is related to the primal feasibility
(see \eg,~\cite{devolder:12:dst,necoara:16:ica,nesterov:12:htm}).
Furthermore,
the norm of the subgradient
is upper bounded by the norm of the~\cgm, \ie, 
for any given subgradients
$\phi'(\pL(\x))$ in~\eqref{eq:tnabla}
and $F'(\pL(\x)) := \nabla f(\pL(\x)) + \phi'(\pL(\x)) \in \partial F(\pL(\x))$,
we have
\begin{align}
||F'(\pL(\x))||
&\le ||\nabla f(\x) - \nabla f(\pL(\x))|| + ||\nabla f(\x) + \phi'(\pL(\x))|| 
	\label{eq:dualgrad} \\
&\le 2L||\x - \pL(\x)||
= 2||\tnabla F(\pL(\x))|| 
\nonumber
,\end{align}
where the first inequality uses the triangle inequality
and the second inequality uses \eqref{eq:L} and~\eqref{eq:tnabla}.
This inequality provides a close relationship between
the primal feasibility
and the dual~\cgm.
Therefore, we next analyze the worst-case bound of the~\cgm of PGM;
Sec.~\ref{sec:pep,spgrad} below discusses
a first-order algorithm that is optimized 
with respect to the~\cgm.\footnote{
One could develop a first-order algorithm
that is optimized with respect to the norm of the subgradient
(rather than its upper bound in Sec.~\ref{sec:pep,spgrad}),
which we leave as future work.
}

The following lemma shows that
PGM monotonically decreases the norm of the~\cgm.

\begin{lemma}
\label{lem:pgmono}
The PGM monotonically decreases the norm of~\cgm, i.e.,
for all \x:
\begin{align}
||\tnabla F(\pL(\x))|| \le ||\tnabla F(\x)||
\label{eq:pg_desc}
.\end{align}
\end{lemma}
\begin{proof}
The proof in~\cite[Lemma 2.4]{necoara:16:ica}
can be easily extended
to prove~\eqref{eq:pg_desc}
using the nonexpansiveness of the proximal mapping (proximity operator)
\cite{combettes:11:psm}.
\end{proof}

\noindent
The following theorem provides a $O(1/N)$ bound 
on the norm of~\cgm for the PGM,
using the idea in~\cite{nesterov:12:htm}
and Lemma~\ref{lem:pgmono}.

\begin{theorem}
Let $F\;:\;\Reals^d\rightarrow\Reals$ be in \cF
and let $\x_0,\cdots,\x_N \in \Reals^d$ be generated by
PGM. Then for $N\ge2$,
\begin{align}
\min_{i\in\{0,\ldots,N\}} ||\tnabla F(\x_i)|| = ||\tnabla F(\x_N)|| 
	\le \frac{2LR}{\sqrt{(N-1)(N+2)}}
\label{eq:pg_pgm}
.\end{align}
\end{theorem}
\begin{proof}
Let $m = \fNh$, and we have
\begin{align*}
\frac{LR^2}{2m} 
	&\stackrel{\eqref{eq:fv_pgm}}{\ge}
		F(\x_m) - F(\x_*)
	\stackrel{\eqref{eq:pg_decr}}{\ge}
		 F(\x_{N+1}) - F(\x_*) + \frac{1}{2L}\sum_{i=m}^N||\tnabla F(\x_i)||^2 \\
	&\stackrel{\eqref{eq:pg_desc}}{\ge} 
		\frac{N-m+1}{2L}||\tnabla F(\x_N)||^2,
\end{align*}
which is equivalent to~\eqref{eq:pg_pgm}
using $m\ge\frac{N-1}{2}$ and $N-m\ge\frac{N}{2}$.
\end{proof}

Despite its inexpensive per-iteration computational cost,
PGM suffers from the slow rate $O(1/N)$
for decreasing both the cost function and the norm of~\cgm.\footnote{
\cite[Thm. 2]{drori:14:pof} and~\cite[Thm. 2]{kim:16:gto} imply that
the $O(1/N)$ rates of 
both the cost function bound~\eqref{eq:fv_pgm}
and the~\cgm norm bound~\eqref{eq:pg_pgm} of PGM
are tight up to a constant respectively.
}
Therefore for acceleration, this paper considers the following class 
of \emph{fixed-step} first-order methods (\FO),
where the $(i+1)$th iteration consists of one proximal gradient evaluation,
just like PGM,
and a weighted summation of previous and current proximal gradient updates
$\{\x_{k+1} - \y_k\}_{k=0}^i$ 
with step coefficients $\{h_{i+1,k}\}_{k=0}^i$.

\fbox{
\begin{minipage}[t]{0.85\linewidth}
\vspace{-10pt}
\begin{flalign}
&\quad \text{\bf Algorithm Class~\FO} & \nonumber \\
&\quad \text{Input: } f\in \cF,\; \x_0\in\Reals^d,\; \y_0 = \x_0. & \nonumber \\
&\quad \text{For } i = 0,\ldots,N-1 & \nonumber \\
&\quad \qquad \x_{i+1} = \pL(\y_i)
			= \y_i - \frac{1}{L}\tnabla F(\y_i) & \nonumber \\
&\quad \qquad \y_{i+1} = \y_i + \sum_{k=0}^i \hkip\,(\x_{k+1} - \y_k)
		= \y_i - \frac{1}{L}\sum_{k=0}^i \hkip\,\tnabla F(\y_k).\!\!\!\!\! & 
		\nonumber
\end{flalign}
\end{minipage}
} \vspace{5pt}

Although the weighted summation in~\FO
seems at first to be inefficient
both computationally and memory-wise,
the optimized~\FO presented in this paper
have equivalent recursive forms that
have memory and computation requirements
that are similar to PGM.
Note that this class~\FO includes PGM
but excludes accelerated algorithms in
\cite{ghadimi:16:agm,nesterov:05:smo,nesterov:13:gmf}
that combine the proximal operations
and the gradient steps in other ways.

Among~\FO\footnote{
The step coefficients of~\FO for FPGM are~\cite{drori:14:pof,kim:16:ofo}
\[ 
h_{i+1,k}
        = \begin{cases}
                \frac{1}{t_{i+1}}
                        \left(t_k - \sum_{j=k+1}^i h_{j,k}\right), & k=0,\ldots,i-1, \\
                1 + \frac{t_i - 1}{t_{i+1}}, & k=i.
        \end{cases}
\] 
}, 
FISTA~\cite{beck:09:afi},
also known as FPGM, is widely used
since it has computation and memory requirements that are similar to PGM
yet it achieves the optimal $O(1/N^2)$ worst-case rate 
for decreasing the cost function,
using Nesterov's acceleration technique~\cite{nesterov:83:amf,nesterov:04}.

\fbox{
\begin{minipage}[t]{0.85\linewidth}
\vspace{-10pt}
\begin{flalign}
&\quad \text{\bf Algorithm FPGM (FISTA)} & \nonumber \\
&\quad \text{Input: } f\in \cF,\; \x_0\in\Reals^d,\; \y_0 = \x_0,\; t_0 = 1. & \nonumber \\
&\quad \text{For } i = 0,\ldots,N-1 & \nonumber \\
&\quad \qquad \x_{i+1} = \pL(\y_i)
        & \nonumber \\
&\quad \qquad t_{i+1} = \frac{1+\sqrt{1+4t_i^2}}{2} & \label{eq:ti} \\
&\quad \qquad \y_{i+1} = \x_{i+1}
                + \frac{t_i - 1}{t_{i+1}}(\x_{i+1} - \x_i) & \nonumber
\end{flalign}
\end{minipage}
} \vspace{5pt}

FPGM has the following bound 
for the cost function~\cite[Thm. 4.4]{beck:09:afi} for any $N\ge1$:
\begin{align}
F(\x_N) - F(\x_*)
	\le \frac{LR^2}{2t_{N-1}^2}
        \le \frac{2LR^2}{(N+1)^2}
\label{eq:fv_fpgm}
,\end{align}
where the 
parameters $t_i$~\eqref{eq:ti} satisfy
\begin{align}
t_i^2 = \sum_{l=0}^it_l \quad\text{and}\quad t_i \ge \frac{i+2}{2} 
\label{eq:ti_rule}
.\end{align}
Sec.~\ref{sec:pep,cost} 
provides a new proof of the cost function bound~\eqref{eq:fv_fpgm} of FPGM
using a new relaxation of PEP,
and illustrates that this particular acceleration of PGM
results from optimizing a relaxed version of
the cost function form of PEP.
In addition, it is shown in~\cite{beck:09:afi,chambolle:15:otc} 
that FPGM and its bound~\eqref{eq:fv_fpgm} 
generalize to any $t_i$ such that 
$t_0 = 1$ and
$t_i^2 \le t_{i-1}^2 + t_i$ 
for all $i\ge1$
with corresponding bound for any $N\ge1$:
\begin{align}
F(\x_N) - F(\x_*)
        \le \frac{LR^2}{2t_{N-1}^2}
,\end{align}
which includes the choice $t_i = \frac{i+a}{a}$ for any $a\ge2$.
Using our relaxed PEP, Sec.~\ref{sec:pep,cost} 
further describes similar but different generalizations of FPGM
that complement our understanding of FPGM.

We are often interested 
in the worst-case analysis of the norm of the (composite) gradient (mapping)
in addition to that of the cost function,
particularly when dealing with dual problems.
To improve the rate $O(1/N)$ of the gradient norm bound of a gradient method,
Nesterov~\cite{nesterov:12:htm} suggested
performing his fast gradient method (FGM)~\cite{nesterov:83:amf,nesterov:04},
a non-proximal version of FPGM, 
for the first $m$ iterations 
and a gradient method for remaining $N-m$ iterations
for smooth convex problems (when $\phi(\x) = 0$).
Here we extend this idea to the nonsmooth composite convex problem~\eqref{eq:prob}
and use FPGM-\m 
to denote the resulting algorithm.

\fbox{
\begin{minipage}[t]{0.85\linewidth}
\vspace{-10pt}
\begin{flalign}
&\quad \text{\bf Algorithm FPGM-\m} & \nonumber \\
&\quad \text{Input: } f\in \cF,\; \x_0\in\Reals^d,\; \y_0 = \x_0,\; t_0 = 1. & \nonumber \\
&\quad \text{For } i = 0,\ldots,N-1 & \nonumber \\
&\quad \qquad \x_{i+1} = \pL(\y_i)
        & \nonumber \\
&\quad \qquad t_{i+1} = \frac{1+\sqrt{1+4t_i^2}}{2},
                \quad i \le m-1
		& \nonumber \\
&\quad \qquad \y_{i+1} = \begin{cases}
                \x_{i+1} + \frac{t_i - 1}{t_{i+1}}(\x_{i+1} - \x_i),
                & i \le m-1, \\
                \x_{i+1}, & \text{otherwise.}
                \end{cases}
                & \nonumber
\end{flalign}
\end{minipage}
} \vspace{5pt}

\noindent
The following theorem provides a $O(1/N^{\frac{3}{2}})$ 
worst-case bound for the norm of~\cgm of FPGM-\m,
using the idea in~\cite{nesterov:12:htm}
and Lemma~\ref{lem:pgmono}.

\begin{theorem}
Let $F\;:\;\Reals^d\rightarrow\Reals$ be in \cF
and let $\x_0,\cdots,\x_N \in \Reals^d$ be generated by
FPGM-\m for $1\le m\le N$. Then for $N\ge1$,
\begin{align}
\min_{i\in\{0,\ldots,N\}}||\tnabla F(\x_i)|| 
	\le ||\tnabla F(\x_N)||
        \le \frac{2LR}{(m+1)\sqrt{N-m+1}}
\label{eq:pg_fpgmh}
.\end{align}
\end{theorem}
\begin{proof}
We have
\begin{align*}
\frac{2LR^2}{(m+1)^2}
        &\stackrel{\eqref{eq:fv_fpgm}}{\ge}
		F(\x_m) - F(\x_*)
        \stackrel{\eqref{eq:pg_decr}}{\ge} 
		F(\x_{N+1}) - F(\x_*) + \frac{1}{2L}\sum_{i=m}^N||\tnabla F(\x_i)||^2 \\
	&\stackrel{\eqref{eq:pg_desc}}{\ge} 
		\frac{N-m+1}{2L}||\tnabla F(\x_N)||^2
,\end{align*}
which is equivalent to~\eqref{eq:pg_fpgmh}.
\end{proof}

\noindent
As noticed by a reviewer,
when $\m = \fNm$,
the worst-case bound~\eqref{eq:pg_fpgmh}
of the~\cgm roughly has its smallest constant $3\sqrt{3}$
for the rate $O(1/N^{\frac{3}{2}})$,
which is better than the choice $\m = \fNh$
in~\cite{nesterov:12:htm}.

Monteiro and Svaiter~\cite{monteiro:13:aah}
considered a variant of FPGM
that replaces $\pL(\cdot)$ of FPGM by $\pLsig(\cdot)$ for $0<\sigma<1$;
that variant, which we denote FPGM-$\sigma$,
satisfies the $O(1/N^{\frac{3}{2}})$ rate for the~\cgm.
This FPGM-$\sigma$ algorithm satisfies
the following cost function and~\cgm worst-case bounds\footnote{
The bound for $\min_{i\in\{0,\ldots,N\}} ||\tnablasig F(\y_i)||$
of FPGM-$\sigma$
is described in a big-O sense in
\cite[Prop. 5.2(c)]{monteiro:13:aah},
and we further computed the constant
in~\eqref{eq:fpgmsig}
by following the derivation of~\cite[Prop. 5.2(c)]{monteiro:13:aah}.
}
\cite[Prop. 5.2]{monteiro:13:aah}
for $N\ge1$:
\begin{align}
F(\x_N) - F(\x_*) &\le \frac{2LR^2}{\sigma^2N^2}, \\
\min_{i\in\{0,\ldots,N\}} ||\tnablasig F(\y_i)|| 
	&\le \frac{2\sqrt{3}}{\sigma}\sqrt{\frac{1+\sigma}{1-\sigma}}
	\frac{LR}{N^{\frac{3}{2}}}
\label{eq:fpgmsig}
.\end{align}
The worst-case bound~\eqref{eq:fpgmsig} of the~\cgm
has its smallest constant
$\frac{2\sqrt{3}}{\sigma^2}\sqrt{\frac{1+\sigma}{1-\sigma}} \approx 16.2$ 
when $\sigma = \frac{\sqrt{17} - 1}{4} \approx 0.78$,
which makes the bound~\eqref{eq:fpgmsig} about
$\frac{16}{3\sqrt{3}} \approx 3$-times
larger 
than the bound~\eqref{eq:pg_fpgmh} of FPGM-$\paren{\m\!=\!\fNm}$ at best.
However, since FPGM-$\sigma$ does not require one to select
the number of total iterations $N$ in advance unlike FPGM-\m,
the FPGM-$\sigma$ algorithm could be useful in practice,
as discussed further in Sec.~\ref{sec:decr}.
Ghadimi and Lan~\cite{ghadimi:16:agm} 
also showed the $O(1/N^{\frac{3}{2}})$ rate 
for a~\cgm worst-case bound of another variant of FPGM,
but the corresponding algorithm in~\cite{ghadimi:16:agm} requires 
two proximal gradient updates per iteration,
combining the proximal operations and the gradient steps 
in a way that differs from the class~\FO
and could be less attractive in terms of
the per-iteration computational complexity.

FPGM has been used in dual problems
\cite{beck:14:aog,beck:09:fgb,beck:14:afd,goldstein:14:fad};
using FPGM-\m and the algorithms in~\cite{ghadimi:16:agm,monteiro:13:aah}
that guarantee $O(1/N^{\frac{3}{2}})$ rate
for minimizing the norm of the~\cgm
could be potentially useful for solving dual problems.
(Using F(P)GM-\m for (dual) smooth convex problems
was discussed in~\cite{devolder:12:dst,necoara:16:ica,nesterov:12:htm}.)
However, FPGM-\m and the algorithms in~\cite{ghadimi:16:agm,monteiro:13:aah} 
are not necessarily the best possible methods
with respect to the worst-case bound of the norm of the~\cgm.
Therefore, Sec.~\ref{sec:pep,spgrad} seeks 
to optimize the step coefficients of~\FO
for minimizing the norm of the~\cgm
using a relaxed PEP.

The next section first provides 
a new proof of FPGM
using our new relaxation on PEP,
and proposes the new generalized FPGM. 

\section{Relaxation and optimization of the cost function form of PEP}
\label{sec:pep,cost}

\subsection{Relaxation for the cost function form of PEP}
\label{sec:pep,cost1}

For~\FO with given step-size coefficients $\bmh := \{h_{i+1,k}\}$,
in principle the worst-case bound on the cost function
after $N$ iterations
corresponds to the solution of the following PEP problem
\cite{drori:14:pof}:
\begin{align}
\mathcal{B}_{\mathrm{P}}(\bmh,N,d,L,R) :=\;
& \max_{\substack{F\in\cF, \\
	\x_0,\cdots,\x_N\in\Reals^d,\; \x_*\in X_*(F) \\
	\y_0,\cdots,\y_{N-1}\in\Reals^d}}
	F(\x_N) - F(\x_*)
        \label{eq:PEP} \tag{P} \\
        &\st \; \x_{i+1} = \pL(\y_i), \quad i=0,\ldots,N-1, 
		\quad ||\x_0 - \x_*|| \le R, \nonumber \\
	&\quad\;\;\;	\y_{i+1} = \y_i + \sum_{k=0}^i \hkip (\x_{k+1} - \y_k),
                \quad i=0,\ldots,N-2. \nonumber
\end{align}
Since (non-relaxed) PEP problems like~\eqref{eq:PEP} are difficult to solve
due to the (infinite-dimensional) functional constraint on $F$,
Drori and Teboulle~\cite{drori:14:pof} suggested 
(for smooth convex problems)
replacing the functional constraint by
a property of $F$ 
related to the update
such as $\pL(\cdot)$ in~\eqref{eq:PEP}.
Taylor~\etal~\cite{taylor:17:ewc,taylor:17:ssc}
discussed properties of $F$
that can replace the functional constraint of PEP
without strictly relaxing~\eqref{eq:PEP},
and provided tight numerical worst-case analysis
for any given $N$.
However, analytical solutions remain 
unknown for~\eqref{eq:PEP} and most PEP problems.

Instead, this paper proposes an alternate relaxation
that is looser than that in~\cite{taylor:17:ewc,taylor:17:ssc} 
but provides tractable and useful analytical results.
We consider the following property of $F$ involving
the proximal gradient update $\pL(\cdot)$~\cite[Lemma 2.3]{beck:09:afi}:
\begin{align}
F(\x) - F(\pL(\y)) \ge \frac{L}{2}||\pL(\y) - \y||^2
	+ L\Inprod{\y - \x}{\pL(\y) - \y}
\quad \forall \x,\y\in\Reals^d
\label{eq:ineq0} 
\end{align}
to replace the functional constraint on $F$.
In particular, we use the following property:
\begin{multline}
\frac{L}{2}||\pL(\y) - \y||^2 - L\Inprod{\pL(\x) - \x}{\pL(\y) - \y} \\
\le F(\pL(\x)) - F(\pL(\y)) + L\Inprod{\pL(\y) - \y}{\x - \y},
\quad \forall \x,\y\in\Reals^d
\label{eq:ineq}
\end{multline}
that results from replacing $\x$ in~\eqref{eq:ineq0} by $\pL(\x)$.
When $\phi(\x) = 0$,
the property~\eqref{eq:ineq} reduces to
\begin{multline}
\frac{1}{2L}||\nabla f(\y)||^2 - \frac{1}{L}\Inprod{\nabla f(\x)}{\nabla f(\y)} \\
\le f\paren{\x - \frac{1}{L}\nabla f(\x)} 
	- f\paren{\y - \frac{1}{L}\nabla f(\y)}
	- \Inprod{\nabla f(\y)}{\x - \y},
\quad \forall \x,\y\in\Reals^d
\label{eq:ineqq}
.\end{multline}
Note that the relaxation of PEP in
\cite{drori:14:pof,kim:16:gto,kim:16:ofo,kim:17:otc,taylor:17:ssc}
for unconstrained smooth convex minimization ($\phi(\x) = 0$)
uses a well-known property of $f$ in~\cite[Thm. 2.1.5]{nesterov:04}
that differs from~\eqref{eq:ineqq}
and does not strictly relax the PEP as discussed in~\cite{taylor:17:ssc},
whereas our relaxation using
\eqref{eq:ineq} and~\eqref{eq:ineqq}
does not guarantee a tight relaxation of~\eqref{eq:PEP}.
Finding a tight relaxation 
that leads to useful (or even optimal)
algorithms remains an open problem 
for nonsmooth composite convex problems.

Similar to~\cite[Problem (Q$'$)]{drori:14:pof},
we (strictly) relax problem~\eqref{eq:PEP} as follows using
a set of constraint inequalities~\eqref{eq:ineq}
at the points $(\x,\y)=(\y_{i-1},\y_i)$ for $i=1,\ldots,N-1$
and $(\x,\y)=(\x_*,\y_i)$ for $i=0,\ldots,N-1$:
\begin{align}
\mathcal{B}_{\mathrm{P1}}(\bmh,N,d,L,R) :=\;
&\max_{\substack{\G\in\Reals^{N\times d}, \\ \bmdel\in\Reals^N}}
        LR^2\delta_{N-1}
        \nonumber \\
        &\st \; \Tr{\G^\top\A_{i-1,i}(\bmh)\G} \le \delta_{i-1} - \delta_i,
                \quad i=1,\ldots,N-1, 
		\label{eq:pPEP} \tag{P1} \\
	&\quad\;\;\; \Tr{\G^\top\D_i(\bmh)\G + \bmnu \bmu_i^\top\G} \le -\delta_i,
                \quad i=0,\ldots,N-1, \nonumber
\end{align}
for any given unit vector $\bmnu\in\Reals^d$,
by defining
the $(i+1)$th standard basis vector $\bmu_i = \bme_{i+1} \in \Reals^N$,
the matrix $\G = [\g_0,\cdots,\g_{N-1}]^\top \in \Reals^{N\times d}$
and the vector $\bmdel = [\delta_0,\cdots,\delta_{N-1}]^\top \in \Reals^N$,
where 
\begin{align}
\begin{cases}
\g_i := - \frac{1}{||\y_0 - \x_*||}(\pL(\y_i) - \y_i) 
	= \frac{1}{L||\y_0 - \x_*||}\tnabla F(\y_i), & \\
\delta_i := \frac{1}{L||\y_0 - \x_*||^2}(F(\pL(\y_i)) - F(\x_*)), &
\end{cases}
\end{align}
for $i=0,\ldots,N-1,*$. Note that $\g_* = [0,\cdots,0]^\top$, $\delta_* = 0$
and $\Tr{\G^\top\bmu_i\bmu_j^\top\G} = \inprod{\g_i}{\g_j}$ by definition.
The matrices $\A_{i-1,i}(\bmh)$ and $\D_i(\bmh)$ are defined as
\begin{align}
\begin{cases}
\A_{i-1,i}(\bmh) :=
        \frac{1}{2}\bmu_i\bmu_i^\top
        - \frac{1}{2}\bmu_{i-1}\bmu_i^\top
        - \frac{1}{2}\bmu_i\bmu_{i-1}^\top
        + \frac{1}{2}\sum_{k=0}^{i-1}
        \hki (\bmu_i\bmu_k^\top + \bmu_k\bmu_i^\top), &\!\!\!\!\! \\
\D_i(\bmh) := \frac{1}{2}\bmu_i\bmu_i^\top + \frac{1}{2}\sum_{j=1}^i\sum_{k=
0}^{j-1}
                \hkj(\bmu_i\bmu_k^\top + \bmu_k\bmu_i^\top), &\!\!\!\!\!
	\end{cases}
\label{eq:ABCDF}
\end{align}
which results from the inequalities~\eqref{eq:ineq}
at the points $(\x,\y) = (\y_{i-1},\y_i)$ and $(\x,\y) = (\x_*,\y_i)$ respectively.

As in~\cite[Problem (DQ$'$)]{drori:14:pof},
problem~\eqref{eq:pPEP} 
has a dual formulation that one can solve numerically
for any given $N$ 
using a semidefinite program (SDP)
to determine an upper bound on the cost function worst-case bound
for any~\FO:\footnote{
\label{ftdualform}
See Appendix~\ref{appen1} for the derivation of the dual formulation~\eqref{eq:D}
of~\eqref{eq:pPEP}.
}
\begin{align}
&F(\x_N) - F(\x_*) \le
\mathcal{B}_{\mathrm{P}}(\bmh,N,d,L,R) \nonumber\\
\le\; &\mathcal{B}_{\mathrm{D}}(\bmh,N,L,R) :=\;
\min_{\substack{(\bmlam,\bmtau)\in\Lambda, \\ \gamma\in\Reals}}
        \left\{
	\frac{1}{2}LR^2\gamma\;:\;
	\left(\begin{array}{cc}
		\bmS(\bmh,\bmlam,\bmtau) & \frac{1}{2}\bmtau \\
		\frac{1}{2}\bmtau^\top & \frac{1}{2}\gamma
	\end{array}\right)
	\succeq 0
	\right\}
	\label{eq:D} \tag{D} 
,\end{align}
where $\bmlam = [\lambda_1,\cdots,\lambda_{N-1}]^\top \in \Reals_+^{N-1}$,
$\bmtau = [\tau_0,\cdots,\tau_{N-1}]^\top \in \Reals_+^N$,
and
\begin{align}
&\Lambda := \left\{(\bmlam,\bmtau)\in\Reals_+^{2N-1}
		\;:\;
		\begin{array}{l}
		\tau_0 = \lambda_1,\;\; \lambda_{N-1} + \tau_{N-1} = 1, \\
		\lambda_i - \lambda_{i+1} + \tau_i = 0, \; i=1,\ldots,N-2,
		\end{array}
		\right\},
	\label{eq:Lam} \\
&\bmS(\bmh,\bmlam,\bmtau) := \sum_{i=1}^{N-1} \lambda_i\A_{i-1,i}(\bmh)
                + \sum_{i=0}^{N-1}\tau_i\D_i(\bmh)
	\label{eq:S}
.\end{align}
This means that one can compute 
a valid upper bound~\eqref{eq:D} of~\eqref{eq:PEP}
for given step coefficients \bmh using a SDP.
The next two sections 
provide an analytical solution to~\eqref{eq:D}
for FPGM and similarly for our new generalized FPGM,
superseding the use of numerical SDP solvers.

\subsection{Generalized FPGM}
\label{sec:gen,fpm}

We specify a feasible point of~\eqref{eq:D}
that leads to our new generalized form of FPGM.

\begin{lemma}
\label{lem:feas}
For the following step coefficients:
\begin{align}
h_{i+1,k}
        &= \begin{cases}
                \frac{t_{i+1}}{T_{i+1}}
                        \left(t_k - \sum_{j=k+1}^i h_{j,k}\right), & k=0,\ldots,i-1, \\
                1 + \frac{(t_i - 1)t_{i+1}}{T_{i+1}}, & k=i,
        \end{cases} 
	\label{eq:h_gen_fpgm}
\end{align}
the choice of variables:
\begin{align}
\lambda_i &= \frac{T_{i-1}}{T_{N-1}}, \quad\! i=1,\ldots,N-1, \quad
\tau_i = \frac{t_i}{T_{N-1}}, \quad\! i=0,\ldots,N-1, \quad
\gamma = \frac{1}{T_{N-1}},
	\label{eq:par_gen_fpgm}
\end{align}
is a feasible point of~\eqref{eq:D}
for any choice of $t_i$ such that 
\begin{align}
t_0 = 1, \quad
t_i > 0, 
\quad\text{and}\quad t_i^2 \le T_i := \sum_{l=0}^i t_l
\label{eq:cond_gen_fpgm}
.\end{align}
\end{lemma}
\begin{proof}
It is obvious that $(\bmlam,\bmtau)$ in~\eqref{eq:par_gen_fpgm}
with~\eqref{eq:cond_gen_fpgm}
is in $\Lambda$~\eqref{eq:Lam}.
Using~\eqref{eq:ABCDF},
the $(i,k)$th entry of the symmetric matrix $\bmS(\bmh,\bmlam,\bmtau)$ in~\eqref{eq:S}
can be written as 
\begin{align*}
&S_{i,k}(\bmh,\bmlam,\bmtau) \\
= &\begin{cases}
	\frac{1}{2}\big((\lambda_i+\tau_i) h_{i,k} 
		+ \tau_i\sum_{j=k+1}^{i-1} h_{j,k}\big),
		& i=2,\ldots,N-1,\;k=0,\ldots,i-2, \\
	\frac{1}{2}\paren{(\lambda_i+\tau_i) h_{i,i-1} - \lambda_i},
		& i=1,\ldots,N-1,\;k=i-1, \\
	\frac{1}{2}\lambda_{i+1},
		& i=0,\ldots,N-2,\;k=i, \\
	\frac{1}{2},
		& i=N-1,\;k=i,
	\end{cases}
\end{align*}
where each element $S_{i,k}(\bmh,\bmlam,\bmtau)$
corresponds to the coefficient of the term $\bmu_i\bmu_k^\top$
of $\bmS(\bmh,\bmlam,\bmtau)$ in~\eqref{eq:S}.
Then, inserting~\eqref{eq:h_gen_fpgm} and~\eqref{eq:par_gen_fpgm} 
to the above yields
\begin{align*}
&S_{i,k}(\bmh,\bmlam,\bmtau) \\
= &\begin{cases}
        \frac{1}{2}\bigg(\frac{T_i}{T_{N-1}}\frac{t_i}{T_i}
		\paren{t_k - \sum_{j=k+1}^{i-1}h_{j,k}} 
		+ \frac{t_i}{T_{N-1}} &\!\!\!\!\! \sum_{j=k+1}^{i-1}h_{j,k}\bigg), \\
                & i=2,\ldots,N-1,\;k = 0,\ldots,i-2, \\
        \frac{1}{2}\paren{\frac{T_i}{T_{N-1}}\paren{1 + \frac{(t_{i-1} - 1)t_i}{T_i}} 
                - \frac{T_{i-1}}{T_{N-1}}},
                & i=1,\ldots,N-1,\;k=i-1, \\
        \frac{T_i}{2T_{N-1}},
                & i=0,\ldots,N-1,\;k=i.
        \end{cases} \\
= &\begin{cases}
        \frac{t_i t_k}{2T_{N-1}},        
		& i=1,\ldots,N-1,\;k=0,\ldots,i-1, \\
        \frac{T_i}{2T_{N-1}},
                & i=0,\ldots,N-1,\;k=i.
	\end{cases}
\end{align*}
Then, using~\eqref{eq:par_gen_fpgm} and~\eqref{eq:cond_gen_fpgm}, 
we finally show the feasibility condition of~\eqref{eq:D}:
\begin{align*}
\left(\begin{array}{cc}
                \bmS(\bmh,\bmlam,\bmtau) & \frac{1}{2}\bmtau \\
                \frac{1}{2}\bmtau^\top & \frac{1}{2}\gamma
        \end{array}\right)
        = \frac{1}{2T_{N-1}}\paren{\diag{\TT - \tt^2} + \tt\tt^\top}
        \succeq 0
,\end{align*}
where
$\tt = \paren{t_0,\cdots,t_{N-1},1}^\top$
and
$\TT = \paren{T_0,\cdots,T_{N-1},1}^\top$.
\end{proof}

\FO with the step coefficients~\eqref{eq:h_gen_fpgm}
would be both computationally and memory\hyp wise inefficient,
so we next present 
an equivalent recursive form of~\FO with \eqref{eq:h_gen_fpgm},
named Generalized FPGM (GFPGM).

\fbox{
\begin{minipage}[t]{0.86\linewidth}
\vspace{-10pt}
\begin{flalign}
&\quad \text{\bf Algorithm GFPGM} \nonumber \\
&\quad \text{Input: } f\in \cF,\; \x_0\in\Reals^d,\; 
	\y_0 = \x_0,\; t_0 = T_0 = 1. \nonumber \\
&\quad \text{For } i = 0,\ldots,N-1 \nonumber \\
&\quad \qquad \x_{i+1} = \pL(\y_i)
        \nonumber \\
&\quad \qquad \text{Choose } t_{i+1} \text{ s.t. } t_{i+1} > 0 
	\text{ and }t_{i+1}^2 \le T_{i+1} := \sum_{l=0}^{i+1}t_l \nonumber \\
&\quad \qquad \y_{i+1} = \x_{i+1}
                + \frac{(T_i - t_i)t_{i+1}}{t_iT_{i+1}}(\x_{i+1} - \x_i)
                + \frac{(t_i^2 - T_i)t_{i+1}}{t_iT_{i+1}}(\x_{i+1} - \y_i)
		\nonumber
\end{flalign}
\end{minipage}
} \vspace{5pt}

\begin{proposition}
\label{prop:fpgm}
The sequence $\{\x_0,\cdots,\x_N\}$ generated
by~\FO with step sizes
\eqref{eq:h_gen_fpgm} is identical
to the corresponding sequence generated by GFPGM.
\end{proposition}
\begin{proof}
See Appendix~\ref{appen2}.
\end{proof}

Using Lemma~\ref{lem:feas}, the following theorem bounds the cost function
for the GFPGM iterates.

\begin{theorem}
\label{thm:gen_fpgm}
Let $F\;:\;\Reals^d\rightarrow\Reals$ be in \cF
and let $\x_0,\cdots,\x_N \in \Reals^d$ be generated by
GFPGM. Then for $N\ge1$,
\begin{align}
F(\x_N) - F(\x_*)
        \le \frac{LR^2}{2T_{N-1}}
\label{eq:fv_gen_fpgm}
.\end{align}
\end{theorem}
\begin{proof}
Using~\eqref{eq:D}, Lemma~\ref{lem:feas}
and Prop.~\ref{prop:fpgm},
we have
\begin{align}
F(\x_N) - F(\x_*) \le
\mathcal{B}_{\mathrm{D}}(\bmh,N,L,R)
= \frac{1}{2}LR^2\gamma
= \frac{LR^2}{2T_{N-1}}
\label{eq:fvbound_gen_fpgm}
.\end{align}
\end{proof}

\noindent
The GFPGM and Thm.~\ref{thm:gen_fpgm} 
reduce to FPGM and~\eqref{eq:fv_fpgm}
when $t_i^2 = T_i$ for all $i$,
and Sec.~\ref{sec:fista} describes that
FPGM results from optimizing the step coefficients of~\FO
with respect to the cost function form of the relaxed PEP~\eqref{eq:D}.
This GFPGM also includes 
the choice $t_i = \frac{i+a}{a}$
for any $a\ge2$ as used in~\cite{chambolle:15:otc},
which we denote as FPGM-$a$
that differs from the algorithm in~\cite{chambolle:15:otc}.
The following corollary provides 
a cost function worst-case bound for FPGM-$a$.

\begin{corollary}
\label{cor:cost}
Let $F\;:\;\Reals^d\rightarrow\Reals$ be in \cF
and let $\x_0,\cdots,\x_N \in \Reals^d$ be generated by
GFPGM with $t_i = \frac{i+a}{a}$ (FPGM-$a$)
for any $a\ge2$. Then for $N\ge1$,
\begin{align}
F(\x_N) - F(\x_*)
        \le \frac{aLR^2}{N(N+2a-1)}
\label{eq:fv_fpgm_}
.\end{align}
\end{corollary}
\begin{proof}
Thm.~\ref{thm:gen_fpgm} implies~\eqref{eq:fv_fpgm_}, since
$t_i = \frac{i+a}{a}$ satisfies~\eqref{eq:cond_gen_fpgm}, \ie,
\begin{align}
T_i - t_i^2 = \frac{(i+1)(i+2a)}{2a} - \frac{(i+a)^2}{a^2} 
	= \frac{(a-2)i^2 + a(2a-3)i}{2a^2}
	\ge 0
\label{eq:Ti_ti}
\end{align}
for any $a\ge2$ and all $i\ge0$.
\end{proof}

\subsection{Related work of GFPGM}

This section shows
that the GFPGM has a close connection to the accelerated algorithm
in~\cite{nesterov:05:smo}
that was developed specifically for a constrained smooth convex problem
with a closed convex set $Q$,
\ie,
\begin{align}
\phi(\x) = \I_Q(\x) := \begin{cases}
0, & \x \in Q, \\
\infty, & \text{otherwise.}
\end{cases}
\end{align}
The projection operator
$\P_Q(\x) := \argmin{\y\in Q} \allowbreak ||\x - \y||$
is used for the proximal gradient update~\eqref{eq:pm}.

We show that the GFPGM
can be written in the following equivalent form,
named GFPGM$'$,
which is similar to that of the accelerated algorithm in~\cite{nesterov:05:smo}
shown below.
Note that the accelerated algorithm in~\cite{nesterov:05:smo}
satisfies the bound~\eqref{eq:fv_gen_fpgm} of the GFPGM
in~\cite[Thm. 2]{nesterov:05:smo}
when $\phi(\x) = \I_Q(\x)$.

\fbox{
\begin{minipage}[t]{0.85\textwidth}
\vspace{-10pt}
\begin{flalign*}
&\quad \text{\bf Algorithm GFPGM$'$} & \\
&\quad \text{Input: } f\in \cF,\; \x_0\in\Reals^d,\;
        \y_0 = \x_0,\; t_0 = T_0 = 1. & \\
&\quad \text{For } i = 0,\ldots,N-1 & \\
&\quad \qquad \x_{i+1} = \pL(\y_i) = \y_i - \frac{1}{L} \tnabla F(\y_i) & \\
&\quad \qquad \z_{i+1} = \y_0 - \frac{1}{L}\sum_{k=0}^i t_k \tnabla F(\y_k) & \\
&\quad \qquad \text{Choose } t_{i+1} \text{ s.t. }
                t_{i+1} > 0 \text{ and }
                t_{i+1}^2 \le T_{i+1} = \sum_{l=0}^{i+1}t_l \\
&\quad \qquad \y_{i+1} = \paren{1 - \frac{t_{i+1}}{T_{i+1}}}\x_{i+1}
                + \frac{t_{i+1}}{T_{i+1}}\z_{i+1} &
\end{flalign*}
\end{minipage}
} \vspace{5pt}

\fbox{
\begin{minipage}[t]{0.85\textwidth}
\vspace{-10pt}
\begin{flalign*}
&\quad \text{\bf Algorithm \cite{nesterov:05:smo} for $\phi(\x) = \I_Q(\x)$} & \\
&\quad \text{Input: } f\in \cF,\; \x_0\in\Reals^d,\;
        \y_0 = \x_0,\; t_0 = T_0 = 1. & \\
&\quad \text{For } i = 0,\ldots,N-1 & \\
&\quad \qquad \x_{i+1} = \pL(\y_i) = \P_Q\paren{\y_i - \frac{1}{L} \nabla f(\y_i)} & \\
&\quad \qquad \z_{i+1} = \P_Q\paren{\y_0 - \frac{1}{L}\sum_{k=0}^i t_k \nabla f(\y_k)} & \\
&\quad \qquad \text{Choose } t_{i+1} \text{ s.t. }
                t_{i+1} > 0 \text{ and }
                t_{i+1}^2 \le T_{i+1} = \sum_{l=0}^{i+1}t_l \\
&\quad \qquad \y_{i+1} = \paren{1 - \frac{t_{i+1}}{T_{i+1}}}\x_{i+1}
                + \frac{t_{i+1}}{T_{i+1}}\z_{i+1} &
\end{flalign*}
\end{minipage}
} \vspace{5pt}

\begin{proposition}
\label{prop:fpgm_}
The sequence $\{\x_0,\cdots,\x_N\}$ generated
by GFPGM is identical
to the corresponding sequence generated by GFPGM$'$.
\end{proposition}
\begin{proof}
See Appendix~\ref{appen3}.
\end{proof}

\noindent
Clearly GFPGM$'$ 
and the accelerated algorithm in~\cite{nesterov:05:smo}
are equivalent for the unconstrained smooth convex problem ($Q = \Reals^d$).
However, when the operation $\P_Q(\x)$ is relatively expensive,
our GFPGM and GFPGM$'$
that use one projection per iteration
could be preferred
over the accelerated algorithm in~\cite{nesterov:05:smo} 
that uses two projections per iteration.

\subsection{Optimizing step coefficients of~\FO
using the cost function form of PEP}
\label{sec:fista}

To find the step coefficients in the class~\FO
that are optimal in terms of the cost function form of PEP,
we would like to solve the following problem:
\begin{align}
\hat{\bmh}_{\mathrm{P}}
        := \argmin{\bmh\in\Reals^{N(N+1)/2}} \mathcal{B}_{\mathrm{P}}(\bmh,N,d,L,R)
	\tag{HP}
\label{eq:HP}
.\end{align}
Because~\eqref{eq:HP} seems intractable,
we instead optimize the step coefficients using the relaxed bound in~\eqref{eq:D}:
\begin{align}
\hat{\bmh}_{\mathrm{D}}
        := \argmin{\bmh\in\Reals^{N(N+1)/2}} \mathcal{B}_{\mathrm{D}}(\bmh,N,L,R)
	\tag{HD}
\label{eq:HD}
.\end{align}
The problem~\eqref{eq:HD} is bilinear,
and a convex relaxation technique in~\cite[Thm. 3]{drori:14:pof}
makes it solvable using numerical methods.
We optimized~\eqref{eq:HD} 
numerically for many choices of $N$
using a SDP solver
\cite{cvxi,gb08}
and based on our numerical results (not shown)
we conjecture that the feasible point in Lemma~\ref{lem:feas}
with $t_i^2 = T_i$ 
that corresponds to FPGM (FISTA)
is a global minimizer of~\eqref{eq:HD}.
It is straightforward to show that
the step coefficients in Lemma~\ref{lem:feas} with $t_i^2 = T_i$ give
the smallest bound of~\eqref{eq:D} and~\eqref{eq:fv_gen_fpgm}
among all feasible points in Lemma~\ref{lem:feas},
but showing optimality among all possible feasible points of~\eqref{eq:HD}
may require further derivations 
as in~\cite[Lemma~3]{kim:16:ofo} using KKT conditions,
which we leave as future work.

This section has provided a new worst-case bound proof of FPGM
using the relaxed PEP,
and suggested that FPGM corresponds to
\FO with optimized step coefficients
using the cost function form of the relaxed PEP.
The next section provides 
a different optimization of the step coefficients of~\FO
that targets the norm of the~\cgm,
because 
minimizing the norm of the composite gradient mapping
is important 
in dual problems 
(see \cite{devolder:12:dst,necoara:16:ica,nesterov:12:htm} and~\eqref{eq:dualgrad}).

\section{Relaxation and optimization of the~\cgm form of PEP}
\label{sec:pep,spgrad}

\subsection{Relaxation for the~\cgm form of PEP}
\label{sec:pep,spgrad1}

To form a worst-case bound on the norm of the~\cgm
for a given \bmh of~\FO,
we use the following PEP
that replaces $F(\x_N) - F(\x_*)$ in~\eqref{eq:PEP} by
the norm squared of the~\cgm.
Here, we consider the smallest~\cgm norm squared 
among all iterates\footnote{
\label{ft}
See Appendix~\ref{appen4}
for the discussion on the choice of $\Om$.
}
($\min_{\x\in\Om} ||L\,(\pL(\x) - \x)||^2 
= \min_{\x\in\Om} ||\tnabla F(\x)||^2$ 
where $\Om := \{\y_0,\cdots,\y_{N-1},\x_N\}$)
as follows:
\begin{align}
\mathcal{B}_{\mathrm{P'}}(\bmh,N,d,L,R) :=\;
& \max_{\substack{F\in\cF, \\
        \x_0,\cdots,\x_N\in\Reals^d,\; \x_*\in X_*(F), \\
	\y_0,\cdots,\y_{N-1}\in\Reals^d}}
        \min_{\x\in\Om}
	||L\,(\pL(\x) - \x)||^2
        \nonumber \\ 
        &\st \; \x_{i+1} = \pL(\y_i), \quad i=0,\ldots,N-1, 
	\quad ||\x_0 - \x_*|| \le R,
	\label{eq:PEP__} \tag{P$'$} \\
        &\quad\;\;\;    \y_{i+1} = \y_i + \sum_{k=0}^i \hkip (\x_{k+1} - \y_k),
                \quad i=0,\ldots,N-2. \nonumber 
\end{align}

Because this infinite-dimensional max-min problem appears intractable,
similar to the relaxation from~\eqref{eq:PEP} to~\eqref{eq:pPEP},
we relax~\eqref{eq:PEP__} to a finite-dimensional problem 
with
an additional constraint
resulting from~\eqref{eq:pg_decr}
that is equivalent to
\begin{align}
\frac{L}{2}||\pL(\x_N) - \x_N||^2 \le F(\x_N) - F(\x_*)
\label{eq:xNp_bound}
\end{align}
and conditions that are equivalent 
to $\alpha \le ||L\,(\pL(\x) - \x)||^2$ 
for all $\x \in \Om$
after replacing 
$\min_{\x\in\Om} ||L\,(\pL(\x) - \x)||^2$ 
by $\alpha$
as in~\cite{taylor:17:ssc}.\footnote{
Here, we simply relaxed~\eqref{eq:PEP__} into~\eqref{eq:pPEP__}
in a way that 
is similar to the relaxation from~\eqref{eq:PEP} to~\eqref{eq:pPEP}.
This relaxation resulted in a constructive analytical worst-case analysis 
on the composite gradient mapping
in this section
that is somewhat similar to that on the cost function 
in Section~\ref{sec:pep,cost}.
However, this relaxation on~\eqref{eq:pPEP__} 
turned out to be relatively loose 
compared to the relaxation on~\eqref{eq:pPEP}
(see Sec.~\ref{sec:disc}),
suggesting there is room for improvement 
in the future
with a tighter relaxation.
}
This relaxation leads to
\begin{align}
\mathcal{B}_{\mathrm{P1'}}(\bmh,N,d,L,R) :=\;
&\max_{\substack{\tG\in\Reals^{(N+1)\times d}, \\ \bmdel\in\Reals^N, \; \alpha\in\Reals}}
        L^2R^2\alpha
	\nonumber \\ 
        &\st \; \Tr{\tG^\top\bA_{i-1,i}(\bmh)\tG} \le \delta_{i-1} - \delta_i,
                \quad i=1,\ldots,N-1, 
		\label{eq:pPEP__} \tag{P1$'$} \\
	&\quad\;\;\; \Tr{\tG^\top\bD_i(\bmh)\tG + \bmnu \bmv_i^\top\tG} \le -\delta_i,
                \quad i=0,\ldots,N-1, \nonumber \\
	&\quad\;\;\; \Tr{\frac{1}{2}\tG^\top\bmv_N\bmv_N^\top\tG}
                \le \delta_{N-1}, \nonumber \\
	&\quad\;\;\; \Tr{-\tG^\top\bmv_i\bmv_i^\top\tG} \le - \alpha,
		\quad i=0,\ldots,N, \nonumber
\end{align}
for any given unit vector $\bmnu\in\Reals^d$,
by defining
the $(i+1)$th standard basis vector $\bmv_i = \bme_{i+1} \in \Reals^{N+1}$,
the matrices 
\begin{align}
\bA_{i-1,i}(\bmh) := \paren{\begin{array}{cc}
                        \A_{i-1,i}(\bmh) & \Zero \\
			\Zero^\top & 0
                \end{array}},
\quad
\bD_i(\bmh) := \paren{\begin{array}{cc}
                        \D_i(\bmh) & \Zero \\
                        \Zero^\top & 0
                \end{array}}
\label{eq:AD_}
\end{align}
where $\Zero = [0,\ldots,0]^\top \in\Reals^N$,
and the matrix $\tG = [\G^\top,\tg_N]^\top \in \Reals^{(N+1)\times d}$
where
\begin{align}
\tg_N := - \frac{1}{||\y_0 - \x_*||}(\pL(\x_N) - \x_N)
        = \frac{1}{L||\y_0 - \x_*||}\tnabla F(\x_N)
.\end{align}

Similar to~\eqref{eq:D}
and~\cite[Problem ($\mathrm{D''}$)]{kim:16:gto},
we have the following dual formulation of~\eqref{eq:pPEP__}
that could be solved using SDP:
\begin{align}
&\mathcal{B}_{\mathrm{D'}}(\bmh,N,L,R) :=\;
\min_{\substack{(\bmlam,\bmtau,\eta,\bmbeta)\in\Lambda', \\ \gamma\in\Reals}}
        \left\{
        \frac{1}{2}L^2R^2\gamma\;:\;
        \left(\begin{array}{cc}
                \bmS'(\bmh,\bmlam,\bmtau,\eta,\bmbeta) 
			& \frac{1}{2}[\bmtau^\top, 0]^\top \\
                \frac{1}{2}[\bmtau^\top,0] & \frac{1}{2}\gamma
        \end{array}\right)
        \succeq 0
        \right\}
        \label{eq:D__} \tag{D$'$}
\end{align}
where $\eta\in\Reals_+$, $\bmbeta = [\beta_0,\cdots,\beta_N]^\top\in\Reals_+^{N+1}$,
and
\begin{align}
&\Lambda' := \left\{(\bmlam,\bmtau,\eta,\bmbeta)\in\Reals_+^{3N+1} 
                \;:\;
		\begin{array}{l}
		\tau_0 = \lambda_1, \;\; \lambda_{N-1} + \tau_{N-1} = \eta, \;\;
		\sum_{i=0}^{N} \beta_i = 1, \\
                \lambda_i - \lambda_{i+1} + \tau_i = 0, \; i=1,\ldots,N-2
		\end{array}\right\},
\label{eq:Lam__} \\
&\bmS'(\bmh,\bmlam,\bmtau,\eta,\bmbeta)
        := \sum_{i=1}^{N-1} \lambda_i\bA_{i-1,i}(\bmh)
                + \sum_{i=0}^{N-1}\tau_i\bD_i(\bmh)
                + \frac{1}{2}\eta\bmv_N\bmv_N^\top
                - \sum_{i=0}^N\beta_i\bmv_i\bmv_i^\top
\label{eq:S__}
.\end{align}
The next section specifies a feasible point of interest
that is in the class of GFPGM
and analyzes the worst-case bound of the norm of the~\cgm.
Then we optimize the step coefficients of~\FO
with respect to the~\cgm form of PEP
leading to a new algorithm that differs from Nesterov's acceleration 
for decreasing the cost function.

\subsection{Worst-case analysis of the~\cgm of GFPGM}

The following lemma provides feasible point of~\eqref{eq:D__}
for the step coefficients~\eqref{eq:h_gen_fpgm} of GFPGM.

\begin{lemma}
\label{lem:feas__}
For the step coefficients $\{h_{i+1,k}\}$ in~\eqref{eq:h_gen_fpgm},
the choice of variables
\begin{align}
\lambda_i &= T_{i-1}\tau_0,
	\quad\!\!\! i=1,\ldots,N-1,
\quad\!\!
\tau_i = \begin{cases}
        \left(\frac{1}{2}\left(\sum_{k=0}^{N-1}\left(T_k - t_k^2\right)
        + T_{N-1}\right)\right)^{-1}, \;\; i = 0,\!\!\!\! \\
        t_i\tau_0, \hspace{93pt} i = 1,\ldots,N-1,\!\!\!\!
        \end{cases} 
\label{eq:par1_gen_fpgm_pg}
\end{align}
\begin{align}
\eta &= T_{N-1} \tau_0, 
\quad
\beta_i = \begin{cases}
        \frac{1}{2}\left(T_i - t_i^2\right)\tau_0, & i=0,\ldots,N-1, \\
        \frac{1}{2}T_{N-1}\tau_0, & i=N,
        \end{cases}
\quad
\gamma = \tau_0.
\label{eq:par2_gen_fpgm_pg}
\end{align}
is a feasible point of~\eqref{eq:D__} for any choice of $t_i$ and $T_i$
satisfying~\eqref{eq:cond_gen_fpgm}.
\end{lemma}
\begin{proof}
It is obvious that $(\bmlam,\bmtau,\eta,\bmbeta)$ 
in~\eqref{eq:par1_gen_fpgm_pg} and~\eqref{eq:par2_gen_fpgm_pg}
with~\eqref{eq:cond_gen_fpgm}
is in $\Lambda'$~\eqref{eq:Lam__}.
Using~\eqref{eq:ABCDF} and~\eqref{eq:AD_},
the $(i,k)$th entry of the symmetric matrix $\bmS'(\bmh,\bmlam,\bmtau,\eta,\bmbeta)$ 
in~\eqref{eq:S__} can be written as
\begin{align*}
&S_{i,k}'(\bmh,\bmlam,\bmtau,\eta,\bmbeta) \\
=& \begin{cases}
        \frac{1}{2}\paren{(\lambda_i + \tau_i) h_{i,k} + \tau_i\sum_{j=k+1}^{i-1}h_{j,k}},
                & i=2,\ldots,N-1,\;k=0,\ldots,i-2, \\
        \frac{1}{2}\paren{(\lambda_i + \tau_i) h_{i,i-1} - \lambda_i},
                & i=1,\ldots,N-1,\;k=i-1, \\
        \frac{1}{2}\lambda_{i+1} - \beta_i,
                & i=0,\ldots,N-2,\;k=i, \\
        \frac{1}{2}\eta - \beta_i,
		& i=N-1,N,\;k=i, \\
	0, 
		& i=N,\;k=0,\ldots,i-1,
        \end{cases}
\end{align*}
and inserting~\eqref{eq:h_gen_fpgm},~\eqref{eq:par1_gen_fpgm_pg},
and~\eqref{eq:par2_gen_fpgm_pg} yields
\begin{align*}
&S_{i,k}'(\bmh,\bmlam,\bmtau,\eta,\bmbeta) \\
= &\begin{cases}
        \frac{1}{2}\bigg(T_i\tau_0\frac{t_i}{T_i}
                \paren{t_k - \sum_{j=k+1}^{i-1}h_{j,k}}
                + t_i\tau_0 &\!\!\!\!\! \sum_{j=k+1}^{i-1}h_{j,k}\bigg), \\
                & i=2,\ldots,N-1,\;k = 0,\ldots,i-2, \\
        \frac{1}{2}\paren{T_i\tau_0\paren{1 + \frac{(t_{i-1} - 1)t_i}{T_i}}
                - T_{i-1}\tau_0},
                & i=1,\ldots,N-1,\;k=i-1, \\
        \frac{1}{2}T_i\tau_0 - \frac{1}{2}(T_i - t_i^2)\tau_0,        
		& i=0,\ldots,N-1,\;k=i, \\
	0,	
		& i=N,\;k=0,\ldots,i,
        \end{cases} \\
= &\begin{cases}
	\frac{1}{2}t_it_k\tau_0, & i=0,\ldots,N-1,\;k=0,\ldots,i, \\
   	0, & i=N,\;k=0,\ldots,i.
    \end{cases}
\end{align*}
Finally, by defining $\ttt = \paren{t_0,\cdots,t_{N-1},0,1}^\top$
we have the feasibility condition of~\eqref{eq:D__}:
\begin{align*}
\left(\begin{array}{cc}
                \bmS'(\bmh,\bmlam,\bmtau,\eta,\bmbeta) 
			& \frac{1}{2}[\bmtau^\top,0]^\top \\
                \frac{1}{2}[\bmtau^\top, 0] & \frac{1}{2}\gamma
        \end{array}\right)
        = \frac{1}{2}\ttt\ttt\tau_0
        \succeq 0
.\end{align*}
\end{proof}

Using Lemma~\ref{lem:feas__}, 
the following theorem bounds the (smallest) norm 
of the composite gradient mapping for the GFPGM iterates.
\begin{theorem}
\label{thm:gen_fpgm_pg}
Let $f\;:\;\Reals^d\rightarrow\Reals$ be in \cF
and let $\x_0,\cdots,\x_N,\y_0,\cdots,\y_{N-1} \allowbreak \in \Reals^d$ 
be generated by GFPGM. 
Then for $N\ge1$,
\begin{align}
\min_{i\in\{0,\ldots,N\}} ||\tnabla F(\x_i)||
\le 
\min_{\x\in\Om} ||\tnabla F(\x)||
        \le \frac{LR}
                {\sqrt{\sum_{k=0}^{N-1}\left(T_k - t_k^2\right) + T_{N-1}}}
\label{eq:gen_fpgm_pg_conv}
.\end{align}
\end{theorem}
\begin{proof}
Lemma~\ref{lem:pgmono} implies the first inequality of~\eqref{eq:gen_fpgm_pg_conv}.
Using~\eqref{eq:D__}, Lemma~\ref{lem:feas__} and Prop.~\ref{prop:fpgm},
we have
\begin{align*}
\min_{\x\in\Om} ||\tnabla F(\x)||^2 
\le \mathcal{B}_{\mathrm{D'}}(\bmh,N,L,R)
= \frac{1}{2}L^2R^2\gamma
= \frac{L^2R^2}{\sum_{k=0}^{N-1}\left(T_k - t_k^2\right) + T_{N-1}}
,\end{align*}
which is equivalent to~\eqref{eq:gen_fpgm_pg_conv}.
\end{proof}

Although the bound~\eqref{eq:gen_fpgm_pg_conv} is not 
tight
due to the relaxation on PEP,
next two sections show that there exists choices of $t_i$
that provide a rate $O(1/N^{\frac{3}{2}})$
for decreasing the~\cgm,
including the choice 
that optimizes the~\cgm form of PEP.

FGM for smooth convex minimization
was shown to achieve the rate $O(1/N^{\frac{3}{2}})$ 
for the decrease of the usual gradient
in~\cite{kim:16:gto}.
In contrast, Thm.~\ref{thm:gen_fpgm_pg} 
provides only a $O(1/N)$ bound for FPGM
(or GFPGM with $t_i$~\eqref{eq:ti})
on the decrease of the~\cgm
since $T_i = t_i^2$ for all $i$ 
and the value of $T_{N-1}$ is $O(N^2)$ for $t_i$~\eqref{eq:ti}.
Sec.~\ref{sec:disc} below numerically studies a tight bound on
the~\cgm of FPGM 
and illustrates that it has a rate that is faster than 
the rate $O(1/N)$ of Thm.~\ref{thm:gen_fpgm_pg},
indicating there is a room for improvement in the~\cgm form of the relaxed PEP.

\subsection{Optimizing step coefficients of~\FO
using the~\cgm form of PEP}
\label{sec:fpgmocg}

To optimize the step coefficients in the class~\FO
in terms of the~\cgm form of the relaxed PEP~\eqref{eq:D__},
we would like to solve the following problem:
\begin{align}
\hat{\bmh}_{\mathrm{D'}}
        := \argmin{\bmh\in\Reals^{N(N+1)/2}} \mathcal{B}_{\mathrm{D'}}(\bmh,N,L,R)
\tag{HD$'$}
\label{eq:HD__}
.\end{align}
Similar to~\eqref{eq:HD},
we use a convex relaxation~\cite[Thm. 3]{drori:14:pof}
to make the bilinear problem~\eqref{eq:HD__}
solvable using numerical methods.
We then numerically optimized~\eqref{eq:HD__}
for many choices of $N$ using a SDP solver
\cite{cvxi,gb08}
and found that the following choice of $t_i$:
\begin{align}
t_i &= \begin{cases}
        1, & i = 0, \\
        \frac{1 + \sqrt{1 + 4t_{i-1}^2}}{2}, & i = 1, \ldots, \fNh - 1, \\
        \frac{N-i+1}{2}, & i = \fNh, \ldots, N-1,
        \end{cases}
\label{eq:ti_rule__}
\end{align}
makes the feasible point in Lemma~\ref{lem:feas__} optimal
empirically with respect to the relaxed bound~\eqref{eq:HD__}.
Interestingly, whereas the usual $t_i$ factors 
(such as~\eqref{eq:ti} and $t_i=\frac{i+a}{a}$ for any $a\ge2$)
increase with $i$ indefinitely,
here, the factors begin decreasing after $i=\fNh-1$.

We also noticed numerically that finding the $t_i$
that minimizes the bound~\eqref{eq:gen_fpgm_pg_conv},
\ie,
solving the following constrained quadratic problem:
\begin{align}
\max_{\{t_i\}} \left\{ 
                \sum_{k=0}^{N-1}\paren{\sum_{l=0}^kt_l - t_k^2}
                + \sum_{l=0}^{N-1}t_l\right\}
\quad\st\quad
t_i \text{ satisfies}~\eqref{eq:cond_gen_fpgm}
\text{ for all } i,
\label{eq:quad}
\end{align}
is equivalent to optimizing~\eqref{eq:HD__}.
This means that the solution of~\eqref{eq:quad}
numerically appears equivalent to~\eqref{eq:ti_rule__},
the (conjectured) solution of~\eqref{eq:HD__}.
Interestingly, 
the unconstrained maximizer of~\eqref{eq:quad} 
without the constraint~\eqref{eq:cond_gen_fpgm}
is $t_i = \frac{N-i+1}{2}$,
and this partially appears
in the constrained maximizer~\eqref{eq:ti_rule__} of the problem~\eqref{eq:quad}.

Based on this numerical evidence, we conjecture that 
the solution $\hat{\bmh}_{\mathrm{D'}}$
of problem
\eqref{eq:HD__}
corresponds to~\eqref{eq:h_gen_fpgm} with~\eqref{eq:ti_rule__}.
Using Prop.~\ref{prop:fpgm},
the following GFPGM form with~\eqref{eq:ti_rule__}
is equivalent to~\FO with the step coefficients~\eqref{eq:h_gen_fpgm} 
for~\eqref{eq:ti_rule__}
that are optimized step coefficients of~\FO
with respect to the norm of the~\cgm,
which we name FPGM-\OPG (\OPG for optimized over~\cgm). 

\fbox{
\begin{minipage}[t]{0.85\linewidth}
\vspace{-10pt}
\begin{flalign*}
&\quad \text{\bf Algorithm FPGM-\OPG (GFPGM with $t_i$ in~\eqref{eq:ti_rule__})} & \\
&\quad \text{Input: } f\in C_L^{1,1}(\Reals^d)\text{ convex},\; \x_0\in\Reals^d,\;
        \y_0 = \x_0,\; t_0 = T_0 = 1. & \\
&\quad \text{For } i = 0,\ldots,N-1 & \\
&\quad \qquad \x_{i+1} = \pL(\y_i) & \\
&\quad \qquad t_{i+1}
                = \begin{cases}
                        \frac{1+\sqrt{1+4t_i^2}}{2}, & i=1,\ldots,\fNh-2, \\
                        \frac{N-i}{2}, & i = \fNh-1,\ldots,N-2,
                        \end{cases} & \\
&\quad \qquad \y_{i+1} = \x_{i+1}
                + \frac{(T_i - t_i)t_{i+1}}{t_iT_{i+1}}(\x_{i+1} - \x_i) \\
&\quad \qquad \hspace{53pt}
	+ \frac{(t_i^2 - T_i)t_{i+1}}{t_iT_{i+1}}(\x_{i+1} - \y_i),
		\quad i < N-1
\end{flalign*}
\end{minipage}
} \vspace{5pt}

The following theorem bounds the cost function 
and the (smallest) norm of the \cgm
for the FPGM-\OPG iterates.

\begin{theorem}
\label{thm:fpgm_pg}
Let $F\;:\;\Reals^d\rightarrow\Reals$ be in \cF
and let $\x_0,\cdots,\x_N,\y_0,\cdots,\y_{N-1} \in \Reals^d$ 
be generated by
FPGM-\OPG. Then for $N\ge1$,
\begin{align}
F(\x_N) - F(\x_*)
	\le \frac{4L||\x_0 - \x_*||^2}{N(N+4)},
\label{eq:fv_fpgm_pg}
\end{align}
and for $N\ge3$,
\begin{align}
\min_{i\in\{0,\ldots,N\}} ||\tnabla F(\x_i)||
\le 
\min_{\x\in\Om} ||\tnabla F(\x)|| 
	\le \frac{2\sqrt{6}LR}{N\sqrt{N-2}}
\label{eq:fpgm_pg_conv}
.\end{align}
\end{theorem}
\begin{proof}
FPGM-\OPG is an instance of the GFPGM,
and thus Thm.~\ref{thm:gen_fpgm} implies \eqref{eq:fv_fpgm_pg}
using
\begin{align*}
T_{N-1} &= T_{m-1} + \sum_{k=m}^{N-1} t_k
        = t_{m-1}^2 + \sum_{k=m}^{N-1} \frac{N-k+1}{2} 
	= t_{m-1}^2 + \sum_{k'=2}^{N-m+1} \frac{k'}{2} \\
        &\ge \frac{(m+1)^2}{4} + \frac{(N-m+1)(N-m+2)}{4} - \frac{1}{2}
        \ge \frac{2N^2+8N+1}{16}
,\end{align*}
where $m = \fNh \ge \frac{N-1}{2}$, $N-m\ge\frac{N}{2}$,
and $T_{m-1} = t_{m-1}^2 \ge \frac{(m+1)^2}{4}$~\eqref{eq:ti_rule}.

In addition, Thm.~\ref{thm:gen_fpgm_pg} implies~\eqref{eq:fpgm_pg_conv},
using
\begin{align}
\sum_{k=0}^{N-1} \left(T_k - t_k^2\right) + T_{N-1}
	\ge \frac{1}{24}(N-2)N^2
\label{eq:app5}
,\end{align}
which we prove in the Appendix~\ref{appen5}.
\end{proof}

The~\cgm bound~\eqref{eq:fpgm_pg_conv} of FPGM-\OPG
is asymptotically $\frac{2\sqrt{2}}{3}$-times smaller 
than the bound~\eqref{eq:pg_fpgmh} of FPGM-$\paren{\m\!=\!\fNm}$.
In addition, the cost function bound~\eqref{eq:fv_fpgm_pg} of FPGM-\OPG
satisfies the optimal rate $O(1/N^2)$,
although the bound~\eqref{eq:fv_fpgm_pg}
is two-times larger than the analogous bound~\eqref{eq:fv_fpgm} of FPGM.

\subsection{Decreasing the~\cgm with a rate $O(1/N^{\frac{3}{2}})$ 
without selecting $N$ in advance}
\label{sec:decr}

FPGM-\OPG and FPGM-\m satisfy a fast rate $O(1/N^{\frac{3}{2}})$
for decreasing the norm of the~\cgm
but require one to select the total number of iterations $N$ in advance,
which could be undesirable in practice.
One could use FPGM-$\sigma$ in~\cite{monteiro:13:aah}
that does not require
selecting $N$ in advance,
but instead we suggest a new choice of $t_i$ in GFPGM
that 
satisfies a~\cgm bound that 
is lower than the bound~\eqref{eq:fpgmsig} of FPGM-$\sigma$.

Based on Thm.~\ref{thm:gen_fpgm_pg},
the following corollary shows 
that GFPGM with $t_i = \frac{i+a}{a}$ (FPGM-$a$)
for any $a>2$ 
satisfies the rate $O(1/N^{\frac{3}{2}})$ of the norm of the~\cgm
without selecting $N$ in advance.
(Cor.~\ref{cor:cost} showed that FPGM-$a$ for any $a\ge2$
satisfies the optimal rate $O(1/N^2)$ 
of the cost function.)

\begin{corollary}
Let $f\;:\;\Reals^d\rightarrow\Reals$ be in \cF
and let $\x_0,\cdots,\x_N,\y_0,\cdots,\y_{N-1} \allowbreak \in \Reals^d$ 
be generated by GFPGM
with $t_i = \frac{i+a}{a}$ (FPGM-$a$)
for any $a\ge2$.
Then for $N\ge1$,
we have the following bound on the (smallest)~\cgm:
\begin{align}
\min_{i\in\{0,\ldots,N\}} ||\tnabla F(\x_i)||
&\le 
\min_{\x\in\Om} ||\tnabla F(\x)|| 
	\nonumber \\
        &\le \frac{a\sqrt{6}LR}{\sqrt{N((a-2)N^2 + 3(a^2-a+1)N + (3a^2+2a-1))}}
\label{eq:cor_fpgm_pg_conv}
.\end{align}
\end{corollary}
\begin{proof}
With $T_i = \frac{(i+1)(i+2a)}{2a}$ and~\eqref{eq:Ti_ti},
Thm.~\ref{thm:gen_fpgm_pg} implies~\eqref{eq:cor_fpgm_pg_conv}
using
\begin{align*}
&\sum_{k=0}^{N-1}(T_k - t_k^2) + T_{N-1}
        = \sum_{k=0}^{N-1}\paren{\frac{(k+1)(k+2a)}{2a} - \frac{(k+a)^2}{a^2}}
                + \frac{N(N+2a-1)}{2a} \\
        =& \sum_{k=0}^{N-1}\paren{\frac{(a-2)k^2 + a(2a-3)k}{2a^2}}
		+ \frac{N(N+2a-1)}{2a} \\
        =& \frac{N}{2a^2}\paren{\frac{(a-2)(N-1)(2N-1)}{6}
		+ \frac{a(2a-3)(N-1)}{2} + a(N+2a-1)} \\
	=& \frac{N((a-2)N^2 + 3(a^2-a+1)N + (3a^2+2a-1))}{6a^2}
.\end{align*}
\end{proof}

\noindent
FPGM-$a$ for any $a>2$
has a~\cgm bound~\eqref{eq:cor_fpgm_pg_conv}
that is asymptotically $\frac{a}{2\sqrt{a-2}}$-times larger 
than the bound~\eqref{eq:fpgm_pg_conv} of FPGM-\OPG.
This gap reduces to $\sqrt{2}$ at best when $a=4$,
which is clearly better than that of FPGM-$\sigma$.
Therefore, this FPGM-$a$ algorithm will be useful
for minimizing the~\cgm with a rate $O(1/N^{\frac{3}{2}})$
without selecting $N$ in advance.

\section{Discussion}
\label{sec:disc}

\subsection{Summary of analytical worst-case bounds
on the cost function and the~\cgm}

Table~\ref{tab:rate} summarizes 
the \emph{asymptotic} worst-case bounds 
of all algorithms discussed 
in this paper.
(Note that the bounds are not guaranteed to be tight.)
In Table~\ref{tab:rate},
FPGM and FPGM-\OPG provide the best known analytical worst-case bounds
for decreasing the cost function and the~\cgm respectively.
When one does not want to select $N$ in advance
for decreasing the~\cgm,
FPGM-$a$ will be a useful alternative to FPGM-\OPG.

\begin{table}[!ht]
\footnotesize
\centering
\begin{tabular}{|l|c|c|c|}
\hline
\multirow{2}{*}{Algorithm}   
	& \multicolumn{2}{c|}{Asymptotic worst-case bound} 
	& Require selecting \\
\cline{2-3}
& Cost function ($\times LR^2$) & Proximal gradient ($\times LR$)& $N$ in advance \\
\hline
PGM       
  & $\frac{1}{2} N^{-1}$ & $2 N^{-1}$ & No  \\ \hline
FPGM 
  & $\bm{2 N^{-2}}$  & $2 N^{-1}$ & No  \\ \hline
FPGM-$\sigma$ ($0< \sigma < 1$) 
  & $\frac{2}{\sigma^2} N^{-2}$ 
  & $\frac{2\sqrt{3}}{\sigma^2}\sqrt{\frac{1+\sigma}{1-\sigma}} N^{-\frac{3}{2}}$
  & \multirow{2}{*}{No} \\
FPGM-$\paren{\sigma\!=\!0.78}$ 
  & $3.3 N^{-2}$
  & $16.2 N^{-\frac{3}{2}}$
  & \\ \hline
FPGM-$\paren{\m\!=\!\fNm}$  
  & $4.5 N^{-2}$ 
  & $5.2 N^{-\frac{3}{2}}$ 
  & Yes \\ \hline
{\bf FPGM-\OPG}
  & $4 N^{-2}$  & $\bm{4.9 N^{-\frac{3}{2}}}$ & Yes \\ \hline
{\bf FPGM-$a$ ($a > 2$)} 
  & $a N^{-2}$ & $\frac{a\sqrt{6}}{\sqrt{a-2}} N^{-\frac{3}{2}}$ 
  & \multirow{2}{*}{No} \\
{\bf FPGM-$\paren{a\!=\!4}$}
  & $4 N^{-2}$ & $6.9 N^{-\frac{3}{2}}$ & \\
\hline
\end{tabular}
\caption{
Asymptotic worst-case bounds
on the cost function $F(\x_N) - F(\x_*)$
and the norm of the~\cgm $\min_{\x\in\Om}||\tnabla F(\x)||$
of PGM, FPGM, FPGM-$\sigma$, FPGM-\m, FPGM-\OPG, and FPGM-$a$.
(The cost function bound for FPGM-\m in the table corresponds
to the bound for FPGM after $\m$ iterations
because a tight bound for the final $N$th iteration is unknown.
The bound on $\min_{i\in\{0,\ldots,N\}}||\tnablasig F(\y_i)||$
is used for FPGM-$\sigma$.)} 
\label{tab:rate}
\end{table}

\subsection{Tight worst-case bounds on the cost function 
and the smallest~\cgm norm}
\label{sec:disc2}

Since none of the bounds presented in Table~\ref{tab:rate} 
are guaranteed to be tight, 
we modified the code\footnote{
The code in Taylor~\etal~\cite{taylor:17:ewc}
currently does not provide a tight bound of the norm of the~\cgm
(and the subgradient), 
so we simply added a few lines
to compute a tight bound.
}
(using SDP solvers~\cite{Lofberg2004,sturm:99:us1})
in Taylor~\etal~\cite{taylor:17:ewc} 
to compare tight (numerical) bounds
for the cost function and the~\cgm
in Tables~\ref{tab:bound} and~\ref{tab:sgbound} respectively
for $N = 1,2,4,10,20,30,40,47,50$. 
This numerical bound is guaranteed to be tight 
when the large-scale condition is satisfied~\cite{taylor:17:ewc}.
Taylor~\etal~\cite[Fig.~1]{taylor:17:ewc} 
already studied a tight worst-case bound 
on the cost function decrease of FPGM numerically,
and found that 
the analytical bound~\eqref{eq:fv_fpgm}
is asymptotically tight.
Table~\ref{tab:bound}
additionally provides numerical tight bounds on the cost function of all algorithms
presented in this paper, 
also suggesting that our relaxation of the cost function form of the PEP
from~\eqref{eq:PEP} to~\eqref{eq:D}
is asymptotically tight (for some algorithms).
In addition, the trend of the tight bounds of the~\cgm in Table~\ref{tab:sgbound}
follows that of the bounds in Table~\ref{tab:rate}.
However, there is gap between them
that is not asymptotically tight,
unlike the gap between the bounds of the cost function 
in Tables~\ref{tab:rate} and~\ref{tab:bound}.
In particular, the numerical tight bound for the~\cgm of FPGM in Table~\ref{tab:sgbound}
has a rate faster than the known rate $O(1/N)$ in Thm.~\ref{thm:gen_fpgm_pg}.
We leave reducing this gap
for the bounds on the norm of the~\cgm as future work,
possibly with a tighter relaxation of PEP.
In addition, FPGM-$\paren{m\!=\!\fNm}$
has a numerical tight bound in Table~\ref{tab:sgbound}
that is even slightly better than
that of FPGM-\OPG, 
unlike our expectation from the analytical bounds in Table~\ref{tab:rate}
and Sec.~\ref{sec:fpgmocg}.
This shows room for improvement in optimizing the step coefficients of~\FO
with respect to the~\cgm,
again possibly with a tighter relaxation of PEP.

\begin{table}[!h]
\footnotesize
\centering
\begin{tabular}{|c|L|L|l|l|L|L|}
\hline
\multirow{2}{*}{N}
    & \multirow{2}{*}{PGM}
        & \multirow{2}{*}{FPGM}     & FPGM      & FPGM    & FPGM    & FPGM \\
    &          &          & -$\paren{\sigma\!=\!0.78}$
                                        & -$\paren{\m\!=\!\fNm}$
                                                & -\OPG     & -$\paren{a\!=\!4}$ \\
\hline
1   & $4.00$   & $4.00$    & $2.43$   & $4.00$   & $4.00$   & $4.00$   \\
2   & $8.00$   & $8.00$    & $4.87$   & $8.00$   & $8.00$   & $8.00$   \\
4   & $16.00$  & $19.35$   & $11.77$  & $17.13$  & $17.60$  & $17.23$   \\
10  & $40.00$  & $79.07$   & $48.11$  & $56.47$  & $59.25$  & $55.88$  \\
20  & $80.00$  & $261.66$  & $159.19$ & $163.75$ & $170.10$ & $159.17$  \\
30  & $120.00$ & $546.51$  & $332.49$ & $321.56$ & $331.97$ & $312.03$  \\
40  & $160.00$ & $932.89$  & $567.57$ & $502.37$ & $544.55$ & $514.73$  \\
47  & $188.00$ & $1263.58$ & $768.76$ & $675.68$ & $723.06$ & $686.33$ \\
50  & $200.00$ & $1420.45$ & $864.20$ & $752.90$ & $807.66$ & $767.37$ \\
\hline
Empi. $O(\cdot)$
    & $N^{-1.00}$ & $N^{-1.89}$ & $N^{-1.89}$
    & $N^{-1.75}$ & $N^{-1.79}$ & $N^{-1.80}$ \\
\hline
Known $O(\cdot)$
    & $N^{-1}$ & $N^{-2}$ & $N^{-2}$
    & $N^{-2}$ & $N^{-2}$ & $N^{-2}$ \\
\hline
\end{tabular}
\caption{
Tight worst-case bounds
on the cost function
$\Frac{LR^2}{(F(\x_N) - F(\x_*))}$
of PGM, FPGM, FPGM-$\paren{\sigma\!=\!0.78}$,
FPGM-$\paren{\m\!=\!\fNm}$,
FPGM-\OPG, and FPGM-$\paren{a\!=\!4}$.
We computed empirical rates
by assuming that the bounds follow the form $bN^{−c}$
with constants $b$ and $c$,
and then by estimating $c$ from points $N = 47, 50$.
Note that the corresponding empirical rates
are underestimated due to the simplified exponential model.
}
\label{tab:bound}
\end{table}

\begin{table}[!h]
\footnotesize
\centering
\begin{tabular}{|c|L|L|l|l|L|L|}
\hline
\multirow{2}{*}{N}   
    & \multirow{2}{*}{PGM}      
	& \multirow{2}{*}{FPGM}     & FPGM      & FPGM    & FPGM    & FPGM \\
    &          &          & -$\paren{\sigma\!=\!0.78}$ 
					& -$\paren{\m\!=\!\fNm}$
						& -\OPG     & -$\paren{a\!=\!4}$ \\
\hline
1   & $1.84$   & $1.84$   & $1.18$   & $1.84$   & $1.84$   & $1.84$   \\
2   & $2.83$   & $2.83$   & $1.78$   & $2.83$   & $2.83$   & $2.83$   \\
4   & $4.81$   & $5.65$   & $3.50$   & $5.09$   & $5.21$   & $5.12$   \\
10  & $10.80$  & $13.24$  & $8.74$   & $14.91$  & $15.60$  & $14.76$  \\
20  & $20.78$  & $27.19$  & $18.83$  & $39.70$  & $39.61$  & $29.21$  \\
30  & $30.78$  & $43.49$  & $30.82$  & $64.45$  & $64.40$  & $47.14$  \\
40  & $40.78$  & $61.76$  & $44.39$  & $92.82$  & $91.99$  & $67.82$  \\
47  & $47.77$  & $75.60$  & $54.73$  & $113.92$ & $113.41$ & $83.67$ \\
50  & $50.77$  & $81.78$  & $59.35$  & $123.54$ & $123.17$  & $90.78$ \\
\hline
Empi. $O(\cdot)$
    & $N^{-0.98}$ & $N^{-1.27}$ & $N^{-1.31}$
    & $N^{-1.31}$ & $N^{-1.33}$ & $N^{-1.32}$ \\
\hline
Known $O(\cdot)$
    & $N^{-1}$ & $N^{-1}$ & $N^{-\frac{3}{2}}$
    & $N^{-\frac{3}{2}}$ & $N^{-\frac{3}{2}}$ & $N^{-\frac{3}{2}}$ \\
\hline
\end{tabular}
\caption{
Tight worst-case bounds
on the norm of the~\cgm
$\Frac{LR}{\paren{\min_{\x\in\Om}||\tnabla F(\x)||}}$
of PGM, FPGM, FPGM-$\paren{\sigma\!=\!0.78}$, 
FPGM-$\paren{\m\!=\!\fNm}$, 
FPGM-\OPG, and FPGM-$\paren{a\!=\!4}$.
Empirical rates were computed as described in Table~\ref{tab:bound}.
(The bound for FPGM-$\sigma$ uses 
$\min_{\x\in\Om}||\tnablasig F(\x)||$.)
}
\label{tab:sgbound}
\end{table}

\subsection{Tight worst-case bounds on the final~\cgm}
\label{sec:disc3}

This paper focused on analyzing the worst-case bound
of the \emph{smallest}~\cgm among all iterates
($\min_{\x\in\Om}||\tnabla F(\x)||$) in addition to the cost function,
whereas the~\cgm at the \emph{final} iterate ($||\tnabla F(\x_N)||$)
could be also considered (see Appendix~\ref{appen4}). 
For example,
the~\cgm bounds~\eqref{eq:pg_pgm} and~\eqref{eq:pg_fpgmh} 
for PGM and FPGM-\m
also apply to the final~\cgm,
and using~\eqref{eq:pg_decr}
we can easily derive a (loose) worst-case bound
on the final~\cgm for other algorithms,
\eg,
such a final~\cgm bound for GFPGM is as follows:
\begin{align}
||\tnabla F(\x_N)|| 
	&\stackrel{\eqref{eq:pg_decr}}{\le}
		\sqrt{2L(F(\x_N) - F(\pL(\x_N)))}
        \le 
		\sqrt{2L(F(\x_N) - F(\x_*))} 
	\label{eq:lpg} \\
        &\stackrel{\eqref{eq:fv_gen_fpgm}}{\le} 
		\frac{LR}{\sqrt{T_{N-1}}} \nonumber
.\end{align}
Since the optimal rate for decreasing the cost function is $O(1/N^2)$,
the~\cgm worst-case bound~\eqref{eq:lpg} 
can provide only a rate $O(1/N)$ at best.
For completeness of the discussion,
Table~\ref{tab:lgbound} reports
tight numerical bounds for the \emph{final}~\cgm.
Here, FPGM, FPGM-$\paren{\sigma\!=\!0.78}$, and 
FPGM-$\paren{a\!=\!4}$
have empirical rates of the worst-case bounds
in Table~\ref{tab:lgbound}
that are slower than
those in Table~\ref{tab:sgbound},
unlike the other three including FPGM-\OPG.

\begin{table}[!h]
\footnotesize
\centering
\begin{tabular}{|c|L|L|l|l|L|L|}
\hline
\multirow{2}{*}{N}
    & \multirow{2}{*}{PGM}
        & \multirow{2}{*}{FPGM}     & FPGM      & FPGM    & FPGM    & FPGM \\
    &          &          & -$\paren{\sigma\!=\!0.78}$
                                        & -$\paren{\m\!=\!\fNm}$
                                                & -\OPG     & -$\paren{a\!=\!4}$ \\
\hline
1   & $1.84$   & $1.84$   & $1.18$   & $1.84$   & $1.84$   & $1.84$   \\
2   & $2.83$   & $2.83$   & $1.78$   & $2.83$   & $2.83$   & $2.83$   \\
4   & $4.81$   & $5.65$   & $3.50$   & $5.09$   & $5.21$   & $5.12$   \\
10  & $10.80$  & $12.68$  & $8.41$   & $14.91$  & $15.60$  & $14.76$  \\
20  & $20.78$  & $22.02$  & $14.26$  & $39.65$  & $39.10$  & $25.96$  \\
30  & $30.78$  & $31.26$  & $20.12$  & $64.40$  & $63.40$  & $34.21$  \\
40  & $40.78$  & $40.46$  & $25.97$  & $92.78$  & $90.16$  & $42.39$  \\
47  & $47.77$  & $46.89$  & $30.06$  & $113.92$ & $110.12$ & $48.13$ \\
50  & $50.77$  & $49.65$  & $31.81$  & $123.53$ & $118.99$ & $50.59$ \\
\hline
Empi. $O(\cdot)$
    & $N^{-0.98}$ & $N^{-0.92}$ & $N^{-0.92}$
    & $N^{-1.31}$ & $N^{-1.25}$ & $N^{-0.81}$ \\
\hline
Known $O(\cdot)$
    & $N^{-1}$ & $N^{-1}$ & $N^{-1}$
    & $N^{-\frac{3}{2}}$ & $N^{-1}$ & $N^{-1}$ \\
\hline
\end{tabular}
\caption{
Tight worst-case bounds
on the norm of the final~\cgm
$\Frac{LR}{||\tnabla F(\x_N)||}$
of PGM, FPGM, FPGM-$\paren{\sigma\!=\!0.78}$,
FPGM-$\paren{\m\!=\!\fNm}$,
FPGM-\OPG, and FPGM-$\paren{a\!=\!4}$.
Empirical rates were computed as described in Table~\ref{tab:bound}.
(The bound for FPGM-$\sigma$ uses
$||\tnablasig F(\x_N)||$.)
}
\label{tab:lgbound}
\end{table}

To best of our knowledge,
FPGM-\m (or algorithms that similarly perform accelerated algorithms in the beginning
and run PGM for the remaining iterations)
is known only to have a rate $O(1/N^{\frac{3}{2}})$ in~\eqref{eq:pg_fpgmh}
for decreasing the final~\cgm,
while FPGM-\OPG was also found to      
inherit such fast rate in Table~\ref{tab:lgbound}.
Therefore, searching for first-order methods
that have a worst-case bound on the final~\cgm
that is lower than that of FPGM-\m (and FPGM-\OPG),
and that possibly do not require knowing $N$ in advance
is an interesting open problem.
Note that
a regularization technique in~\cite{nesterov:12:htm}
that provides a faster rate $O(1/N^2)$
(up to a logarithmic factor)
for decreasing the final gradient norm
for smooth convex minimization
can be easily extended for 
rapidly minimizing the final~\cgm with such rate
for the composite problem~\eqref{eq:prob};
however, that approach requires knowing $R$ in advance.

\subsection{Tight worst-case bounds on the final subgradient}
\label{sec:disc4}

This paper has mainly focused
on the norm of the composite gradient mapping
based on~\eqref{eq:dualgrad},
instead of the subgradient norm that is of primary interest
in the dual problem
(see \eg, \cite{devolder:12:dst,necoara:16:ica,nesterov:12:htm}).
Therefore to have a better sense of subgradient norm bounds,
we computed tight numerical bounds 
on the \emph{final}\footnote{
Using modifications of the code in~\cite{taylor:17:ewc}
to compute tight bounds on the \emph{final} subgradient norm
was easier than for the \emph{smallest} subgradient norm
among all iterates.
Even without the \emph{smallest} subgradient norm bounds,
the bounds on the \emph{final} subgradient norm 
in Table~\ref{tab:gbound} (compared to Table~\ref{tab:lgbound})
provide some insights (beyond~\eqref{eq:dualgrad})
on the relationship between 
the bounds on the subgradient norm and the~\cgm norm
as discussed in Sec.~\ref{sec:disc4}.
We leave further modifying the code in~\cite{taylor:17:ewc}
for computing tight bounds on the \emph{smallest} subgradient norm
or other criteria
as future work.
}
subgradient norm $||F'(\x_N)||$
in Table~\ref{tab:gbound}
and compared them with Table~\ref{tab:lgbound}.

For all six algorithms, empirical rates in Table~\ref{tab:gbound}
are similar to those for the final~\cgm in Table~\ref{tab:lgbound}.
In particular, the subgradient norm bounds for 
the three algorithms PGM, FPGM-$\paren{\m\!=\!\fNm}$,
and FPGM-\OPG
are almost identical to those in Table~\ref{tab:lgbound}
except for the first few iterations,
eliminating the concern of using~\eqref{eq:dualgrad}
for such cases. 
On the other hand, 
the other three algorithms 
FPGM, 
FPGM-$\paren{\sigma\!=\!0.78}$,
and FPGM-$\paren{a\!=\!4}$
almost tightly satisfy the inequality~\eqref{eq:dualgrad}
for most $N$,
and thus have bounds on the final subgradient
that are about twice larger than those on the final~\cgm.
Therefore, regardless of~\eqref{eq:dualgrad}, 
Table~\ref{tab:gbound}
further supports the use of FPGM-$\paren{\m\!=\!\fNm}$ and FPGM-\OPG 
over FPGM and other algorithms
in dual problems.

\begin{table}[!h]
\footnotesize
\centering
\begin{tabular}{|c|L|L|l|l|L|L|}
\hline
\multirow{2}{*}{N}
    & \multirow{2}{*}{PGM}
        & \multirow{2}{*}{FPGM}     & FPGM      & FPGM    & FPGM    & FPGM \\
    &          &          & -$\paren{\sigma\!=\!0.78}$    
                                        & -$\paren{\m\!=\!\fNm}$
                                                & -\OPG     & -$\paren{a\!=\!4}$ \\
\hline                                          
1   & $1.00$   & $1.00$   & $0.61$   & $1.00$   & $1.00$   & $1.00$   \\
2   & $2.00$   & $2.00$   & $1.22$   & $2.00$   & $2.00$   & $2.00$   \\
4   & $4.00$   & $4.83$   & $2.94$   & $4.28$   & $4.40$   & $4.31$   \\
10  & $10.00$  & $7.60$   & $4.67$   & $14.12$  & $14.81$  & $12.10$  \\
20  & $20.00$  & $12.58$  & $7.67$   & $38.29$  & $36.65$  & $16.85$  \\
30  & $30.00$  & $17.63$  & $10.74$  & $62.71$  & $60.40$  & $21.61$  \\
40  & $40.00$  & $22.67$  & $13.80$  & $91.00$  & $86.62$  & $26.47$  \\
47  & $47.00$  & $26.20$  & $15.94$  & $112.01$ & $106.21$ & $29.91$ \\
50  & $50.00$  & $27.71$  & $16.86$  & $121.53$ & $114.93$ & $31.39$ \\
\hline
Empi. $O(\cdot)$
    & $N^{-1.00}$ & $N^{-0.91}$ & $N^{-0.91}$
    & $N^{-1.32}$ & $N^{-1.27}$ & $N^{-0.78}$ \\
\hline 
Known $O(\cdot)$
    & $N^{-1}$ & $N^{-1}$ & $N^{-1}$
    & $N^{-\frac{3}{2}}$ & $N^{-1}$ & $N^{-1}$ \\
\hline
\end{tabular}
\caption{
Tight worst-case bounds
on the subgradient norm
$\Frac{LR}{||F'(\x_N)||}$
of PGM, FPGM, FPGM-$\paren{\sigma\!=\!0.78}$,
FPGM-$\paren{\m\!=\!\fNm}$,
FPGM-\OPG, and FPGM-$\paren{a\!=\!4}$,
where $F'(\x) \in \partial F(\x)$ is a subgradient.
Empirical rates were computed as described in Table~\ref{tab:bound}.
}
\label{tab:gbound}
\end{table}

\section{Conclusion}
\label{sec:conc}

This paper analyzed and developed
fixed-step first-order methods (\FO) 
for nonsmooth composite convex cost functions.
We showed an alternate proof of FPGM (FISTA) using PEP,
and suggested that FPGM (FISTA)
results from optimizing the step coefficients of~\FO
with respect to the cost function form of the (relaxed) PEP.
We then described a new generalized version of FPGM
and analyzed its worst-case bound using the (relaxed) PEP
over both the cost function and the norm of the~\cgm.
Furthermore, we
optimized the step coefficients of~\FO
with respect to the~\cgm form of the (relaxed) PEP,
yielding FPGM-\OPG,
which could be useful particularly when tackling dual problems.

Our relaxed PEP provided tractable analysis
of the optimized step coefficients of~\FO
with respect to the cost function and the norm of the~\cgm,
but the relaxation is not guaranteed to be tight and
the corresponding accelerations of PGM (FPGM and FPGM-\OPG) 
are thus unlikely to be optimal.
Therefore, finding optimal step coefficients of~\FO
over the cost function and the norm of the~\cgm
remain as future work.
Nevertheless, the proposed FPGM-\OPG that optimizes 
the~\cgm form of the relaxed PEP
and the FPGM-$a$ (for any $a>2$)
may be useful in dual problems.

\section*{Software}

Matlab codes
for the SDP approaches in 
Sec.~\ref{sec:fista},
Sec.~\ref{sec:fpgmocg}
and Sec.~\ref{sec:disc}
are available at
https://gitlab.eecs.umich.edu/michigan-fast-optimization.

\appendix

\section{Derivation of the dual formulation~\eqref{eq:D} of~\eqref{eq:pPEP}}
\label{appen1}

The derivation below is similar to~\cite[Lemma 2]{drori:14:pof}.

We replace $\max_{\G,\bmdel}\allowbreak LR^2\delta_{N-1}$ of~\eqref{eq:pPEP}
by $\min_{\G,\bmdel}\{-\delta_{N-1}\}$ for convenience in this section.
The corresponding dual function of such~\eqref{eq:pPEP} is then defined as
\begin{align*}
H(\bmlam,\bmtau;\bmh) 
	= \min_{\substack{\G\in\Reals^{N\times d}, \\ \bmdel\in\Reals^N}} 
	\left\{\cL(\G,\bmdel,\bmlam,\bmtau;\bmh)
		 := \cL_1(\bmdel,\bmlam,\bmtau) + \cL_2(\G,\bmlam,\bmtau;\bmh)\right\}
\end{align*}
for dual variables $\bmlam=[\lambda_1,\cdots,\lambda_{N-1}]^\top\in\Reals_+^{N-1}$ 
and $\bmtau=[\tau_0,\cdots,\tau_{N-1}]^\top\in\Reals_+^N$,
where
$\cL(\G,\bmdel,\bmlam,\bmtau;\bmh)$ is a Lagrangian function, and
\begin{align*}
\cL_1(\bmdel,\bmlam,\bmtau) 
	&:= -\delta_{N-1} 
	+ \sum_{i=1}^{N-1}\lambda_i(\delta_i - \delta_{i-1})
	+ \sum_{i=0}^{N1}\tau_i\delta_i, \\
\cL_2(\G,\bmlam,\bmtau;\bmh) 
	&:= \sum_{i=1}^{N-1}\lambda_i\Tr{\G^\top\A_{i-1,i}(\bmh)\G}
	+ \sum_{i=0}^{N-1}\tau_i\Tr{\G^\top\D_i(\bmh)\G + \bmnu\bmu_i^\top\G}.
	\nonumber
\end{align*} 
Here, $\min_{\del}\cL_1(\bmdel,\bmlam,\bmtau) = 0$ for any $(\bmlam,\bmtau)\in\Lambda$
where $\Lambda$ is defined in~\eqref{eq:Lam},
and $\min_{\del}\cL_1(\bmdel,\bmlam,\bmtau)=-\infty$ otherwise.

For any given unit vector $\bmnu$,
\cite[Lemma 1]{drori:14:pof} implies
\begin{align*}
\min_{\G\in\Reals^{N\times d}} \cL_2(\G,\bmlam,\bmtau) 
	= \min_{\bmw\in\Reals^N} \cL_2(\bmw\bmnu^\top,\bmlam,\bmtau)
,\end{align*}
and thus for any $(\bmlam,\bmtau)\in\Lambda$, we can rewrite the dual function as
\begin{align*}
H(\bmlam,\bmtau;\bmh) 
	&= \min_{\bmw\in\Reals^N}\{\bmw^\top\bmS(\bmh,\bmlam,\bmtau)\bmw + \tau^\top\bmw\} \\
	&= \max_{\gamma\in\Reals}\left\{-\frac{1}{2}\gamma\;:\;
		\bmw^\top\bmS(\bmh,\bmlam,\bmtau)\bmw + \bmtau^\top\bmw
		\ge -\frac{1}{2}\gamma,\; \forall\bmw\in\Reals^N\right\}
		\nonumber \\
	&= \max_{\gamma\in\Reals}\left\{-\frac{1}{2}\gamma\;:\;
        \left(\begin{array}{cc}
                \bmS(\bmh,\bmlam,\bmtau) & \frac{1}{2}\bmtau \\
                \frac{1}{2}\bmtau^\top & \frac{1}{2}\gamma
        \end{array}\right)
        \succeq 0
        \right\}
,\end{align*}
where $\bmS(\bmh,\bmlam,\bmtau)$ is defined in~\eqref{eq:S}.
Therefore the dual problem of~\eqref{eq:pPEP} becomes~\eqref{eq:D},
recalling that we previously replaced 
$\max_{\G,\bmdel} LR^2\delta_{N-1}$ of~\eqref{eq:pPEP}
by $\min_{\G,\bmdel} \{-\delta_{N-1}\}$.

\section{Proof of Prop.~\ref{prop:fpgm}}
\label{appen2}

The proof is similar to~\cite[Prop. 2, 3 and 4]{kim:16:ofo}.

We first show that $\{\hkip\}$ in~\eqref{eq:h_gen_fpgm} is equivalent to
\begin{align}
\hkip = \begin{cases}
                \frac{(T_i-t_i)t_{i+1}}{t_iT_{i+1}}\hki
                        & i=0,\ldots,N-1,\;k=0,\ldots,i-2, \\
                \frac{(T_i-t_i)t_{i+1}}{t_iT_{i+1}}(h_{i,i-1} - 1),
                        & i=0,\ldots,N-1,\;k=i-1, \\
                1 + \frac{(t_i - 1)t_{i+1}}{T_{i+1}},
                        & i=0,\ldots,N-1,\;k=i,
        \end{cases}
\label{eq:hh_gen_fpm1}
\end{align}
We use the notation $\hki'$ for the coefficients~\eqref{eq:h_gen_fpgm}
to distinguish from~\eqref{eq:hh_gen_fpm1}.
It is obvious that $\hiip' = \hiip, i=0,\ldots,N-1$,
and we clearly have
\begingroup
\allowdisplaybreaks
\begin{align*}
\himip' &= \frac{t_{i+1}}{T_{i+1}}\paren{t_{i-1} - \himi'}
        = \frac{t_{i+1}}{T_{i+1}}
        \paren{t_{i-1} - \paren{1 + \frac{(t_{i-1}-1)t_i}{T_i}}} \\
        &= \frac{(t_{i-1}-1)(T_i - t_i)t_{i+1}}{T_iT_{i+1}}
        = \frac{(T_i - t_i)t_{i+1}}{t_iT_{i+1}}(\himi - 1)
        = \himip
\end{align*}
\endgroup
We next use induction by assuming $\hkip' = \hkip$
for $i=0,\ldots,n-1,\;k=0,\ldots,i$. We then have
\begingroup
\allowdisplaybreaks
\begin{align*}
\hknp' &= \frac{t_{n+1}}{T_{n+1}}
                \paren{t_k - \sum_{j=k+1}^n\hkj'}
        = \frac{t_{n+1}}{T_{n+1}}
                \paren{t_k - \sum_{j=k+1}^{n-1}\hkj' - \hkn'} \\
        &= \frac{t_{n+1}}{T_{n+1}}
                \paren{\frac{T_n}{t_n}\hkn' - \hkn'}
        = \frac{(T_n - t_n)t_{n+1}}{t_nT_{n+1}}\hkn
        = \hknp
\end{align*}
\endgroup

Next, using~\eqref{eq:hh_gen_fpm1},
we show that~\FO with~\eqref{eq:h_gen_fpgm}
is equivalent to the GFPGM.
We use induction, and for clarity, we use the notation
$\y_0',\cdots,\y_N'$ for~\FO with~\eqref{eq:hh_gen_fpm1}.
It is obvious that $\y_0' = \y_0$, and we have
\begingroup
\allowdisplaybreaks
\begin{align}
\y_1' &= \y_0' - \frac{1}{L}h_{1,0}\tnabla F(\y_0')
        = \y_0 - \frac{1}{L}\paren{1 + \frac{(t_0 - 1)t_1}{T_1}}\tnabla F(\y_0) 
      	\nonumber \\
	&= \x_1 + \frac{(T_0 - t_0)t_1}{t_0T_1}(\x_1 - \x_0)
		+ \frac{(t_0^2 - T_0)t_1}{t_0T_1}(\x_1 - \y_0)
        = \y_1
	\nonumber
,\end{align}
\endgroup
since $T_0 = t_0$.
Assuming $\y_i' = \y_i$ for $i=0,\ldots,n$, we then have
\begingroup
\allowdisplaybreaks
\begin{align*}
&\quad \y_{n+1}'
= \y_n' - \frac{1}{L}h_{n+1,n}\tnabla F(\y_n')
        - \frac{1}{L}h_{n+1,n-1}\tnabla F(\y_{n-1}')
        - \frac{1}{L}\sum_{k=0}^{n-2}h_{n+1,k}\tnabla F(\y_k') \\
&= \y_n - \frac{1}{L}\paren{1 + \frac{(t_n-1)t_{n+1}}
                {T_{n+1}}}\tnabla F(\y_n) \\
        &\quad - \frac{1}{L}\frac{(T_n-t_n)t_{n+1}}{t_nT_{n+1}}
                (h_{n,n-1} - 1)\tnabla F(\y_{n-1})
        - \frac{1}{L}\sum_{k=0}^{n-2}
                \frac{(T_n - t_n)t_{n+1}}{t_nT_{n+1}}
                h_{n,k}\tnabla F(\y_k) \\
&= \x_{n+1} - \frac{1}{L}\frac{(t_n^2 - T_n)t_{n+1}}{t_nT_{n+1}}\tnabla F(\y_n) \\
	&\quad - \frac{1}{L}\frac{(T_n-t_n)t_{n+1}}{t_nT_{n+1}}
        \paren{\tnabla F(\y_n) - \tnabla F(\y_{n-1}) + \sum_{k=0}^{n-1}h_{n,k}\tnabla F(\y_k)} \\
&= \x_{n+1} + \frac{(t_n^2 - T_n)t_{n+1}}{t_nT_{n+1}}(\x_{n+1} - \y_n) \\
        &\quad + \frac{(T_n-t_n)t_{n+1}}{t_nT_{n+1}}
                \paren{-\frac{1}{L}\tnabla F(\y_n) + \frac{1}{L}\tnabla F(\y_{n-1})
                + \y_n - \y_{n-1}} \\
&= \x_{n+1} + \frac{(T_n-t_n)t_{n+1}}{t_nT_{n+1}}(\x_{n+1} - \x_n)
        + \frac{(t_n^2 - T_n)t_{n+1}}{t_nT_{n+1}}(\x_{n+1} - \y_n)
= \y_{n+1}
.\end{align*}
\endgroup

\section{Proof of Prop.~\ref{prop:fpgm_}}
\label{appen3}

The proof is similar to~\cite[Prop. 1 and 5]{kim:16:ofo}.

We use induction, and for clarity, we use the notation
$\y_0',\cdots,\y_N'$ for~\FO with~\eqref{eq:h_gen_fpgm}
that is equivalent to GFPGM by Prop.~\ref{prop:fpgm}.
It is obvious that $\y_0' = \y_0$, and we have
\begin{align}
\y_1' &= \y_0' - \frac{1}{L}h_{1,0}\tnabla F(\y_0')
        = \y_0 - \frac{1}{L}\paren{1 + \frac{(t_0 - 1)t_1}{T_1}}\tnabla F(\y_0)
        \nonumber \\
        &= \paren{1-\frac{t_1}{T_1}}\paren{\y_0 - \frac{1}{L}\tnabla F(\y_0)} 
		+ \frac{t_1}{T_1}\paren{\y_0 - \frac{1}{L}t_0\tnabla F(\y_0)}
	\nonumber \\
	&= \paren{1-\frac{t_1}{T_1}}\x_1 + \frac{t_1}{T_1}\z_1
	= \y_1
        \nonumber
.\end{align}
Assuming $\y_i' = \y_i$ for $i=0,\ldots,n$, we then have
\begingroup
\allowdisplaybreaks
\begin{align*}
&\quad \y_{n+1}'
= \y_n' - \frac{1}{L}h_{n+1,n}\tnabla F(\y_n')
        - \frac{1}{L}\sum_{k=0}^{n-1}h_{n+1,k}\tnabla F(\y_k') \\
&= \y_n - \frac{1}{L}\paren{1 + \frac{(t_n-1)t_{n+1}}
                {T_{n+1}}}\tnabla F(\y_n)
        - \frac{1}{L}\sum_{k=0}^{n-1}
                \frac{t_{n+1}}{T_{n+1}}
                \paren{t_k - \sum_{j=k+1}^n h_{j,k}}
                \tnabla F(\y_k) \\
&= \paren{1 - \frac{t_{n+1}}{T_{n+1}}}\paren{\y_n - \frac{1}{L}\tnabla F(\y_n)} \\
	&\quad + \frac{t_{n+1}}{T_{n+1}}
        \paren{\y_n
                - \frac{1}{L}\sum_{k=0}^n t_k\tnabla F(\y_k)
                + \frac{1}{L}\sum_{k=0}^{n-1}\sum_{j=k+1}^n h_{j,k}\tnabla F(\y_k)} \\
&= \paren{1 - \frac{t_{n+1}}{T_{n+1}}}\paren{\y_n - \frac{1}{L}\tnabla F(\y_n)}
        + \frac{t_{n+1}}{T_{n+1}}
        \paren{\y_0 - \frac{1}{L}\sum_{k=0}^n t_k\tnabla F(\y_k)} \\
&= \paren{1 - \frac{t_{n+1}}{T_{n+1}}}\x_{n+1}
        + \frac{t_{n+1}}{T_{n+1}}\z_{n+1}
.\end{align*}
\endgroup

\section{Discussion on the choice of $\Om$ in Sec.~\ref{sec:pep,spgrad1}}
\label{appen4}

Our formulation~\eqref{eq:PEP__} examines
the set $\Om = \{\y_0,\cdots,\allowbreak\y_{N-1},\x_N\}$
and eventually leads to the best known analytical bound
on the norm of the~\cgm in Thm.~\ref{thm:fpgm_pg}
among fixed-step first-order methods.

An alternative formulation would be to use the set
$\{\y_0,\cdots,\y_{N-1}\}$
(i.e., excluding the point $\x_N$).
For this alternative,
we could simply replace the inequality~\eqref{eq:xNp_bound}
with the condition $0 \le F(\y_{N-1}) - F(\x_*)$ to derive a slightly different relaxation.
(One could use other conditions
at the point $\y_{N-1}$ as in~\cite{taylor:17:ewc}
for a tight relaxation, 
but this is beyond the scope of this paper.)
We found that the corresponding (loose) relaxation~\eqref{eq:pPEP__} 
using $\{\y_0,\cdots,\y_{N-1}\}$
leads to
a larger upper bound than~\eqref{eq:gen_fpgm_pg_conv}
in Thm.~\ref{thm:gen_fpgm_pg}
for the set $\Om$.

Another alternative would be to use the set
$\{\x_0,\cdots,\x_N\}$,
which we leave as future work.
Nevertheless, the inequality in Lemma~\ref{lem:pgmono}
provides a bound for that set $\{\x_0,\cdots,\x_N\}$ as seen
in Thm.~\ref{thm:gen_fpgm_pg} and~\ref{thm:fpgm_pg}.

We could also consider
the final point $\x_N$ (or $\y_N$) in~\eqref{eq:PEP__}
instead of the minimum over a set of points.
However, the corresponding (loose) relaxation~\eqref{eq:pPEP__}
yielded only an $O(1/N)$ bound at best 
(even for the corresponding optimized step coefficients of~\eqref{eq:HD__})
on the final~\cgm norm. 
So we leave finding its tighter relaxation as future work.
Note that Table~\ref{tab:lgbound}
reports tight numerical bounds on the composite gradient mapping norm
at the final point $\x_N$ of algorithms considered.

\section{Proof of Equation~\eqref{eq:app5} in Thm.~\ref{thm:fpgm_pg}}
\label{appen5}

\begingroup
\allowdisplaybreaks
\begin{align*}
&\;\sum_{k=0}^{N-1} \left(T_k - t_k^2\right) + T_{N-1} \\
        =&\; \sum_{k=m}^{N-1} \left(t_{m-1}^2 + \sum_{l=m}^k t_l - t_k^2\right)
                + t_{m-1}^2 + \sum_{l=m}^{N-1} t_l \\
        =&\; (N - m + 1)t_{m-1}^2
                + \sum_{k=m}^{N-1}\left(\sum_{l=m}^k\frac{N-l+1}{2}
                        - \left(\frac{N-k+1}{2}\right)^2\right)
                + \sum_{l=m}^{N-1} \frac{N-l+1}{2} \\
        =&\; (N - m + 1)t_{m-1}^2
                + \sum_{k'=0}^{N-m-1}\left(\sum_{l'=0}^{k'}\frac{N-l'-m+1}{2}
                        - \left(\frac{N-k'-m+1}{2}\right)^2\right) \\
	&\! + \sum_{l'=0}^{N-m-1} \frac{N-l'-m+1}{2}  \\
        =&\; (N - m + 1)t_{m-1}^2 \\
	&\! +\!\!\! \sum_{k=0}^{N-m-1}\!\!\left(\frac{(N-m+1)(k+1)}{2} - \frac{k(k+1)}{4}
                        - \frac{(N-m+1)^2 - 2(N-m+1)k + k^2}{4}\right) \\
        &\! + \frac{(N-m+1)(N-m)}{2} - \frac{(N-m-1)(N-m)}{4}\\
        =&\; (N - m + 1)t_{m-1}^2
                +\!\!\! \sum_{k=0}^{N-m-1}\!\left(
                        - \frac{k^2}{2}
                        + (N-m+3/4)k
                        - \frac{(N-m-1)(N-m+1)}{4}
                        \right) \\
        &\! + \frac{(N-m)(N-m+3)}{4} \\
        =&\; (N - m + 1)t_{m-1}^2
                - \frac{(N-m-1)(N-m-1/2)(N-m)}{6} \\
	&\! + \frac{(N-m-1)(N-m)(N-m+3/4)}{2} \\
        &\! - \frac{(N-m-1)(N-m)(N-m+1)}{4}
                +\frac{(N-m)(N-m+3)}{4} \\
        \ge&\; \frac{(N-m+1)(m + 1)^2}{4}
                + \frac{(N-m-1)(N-m)^2}{3}
                - \frac{(N-m)^2(N-m+1)}{4} \\
        \ge&\; \frac{(N-m-1)(N-m)^2}{3} \\
        \ge&\; \frac{1}{24}(N-2)N^2
,\end{align*}
where
$m=\fNh \ge \frac{N-1}{2}$,
$N-m \ge \frac{N}{2}$,
and $t_{m-1} \ge \frac{m+1}{2}$~\eqref{eq:ti_rule}.
\endgroup

\section*{Acknowledgements}
The authors would like to thank the anonymous referees
for very useful comments that have improved the quality of this paper.

\ifsiopt
\else
	\bibliographystyle{spmpsci}
\fi
\bibliography{master,mastersub}   

\end{document}